\documentclass[a4paper]{amsart}

\usepackage{amsmath,amsthm,amssymb,latexsym,bbm,comment, mathrsfs}
\usepackage{thmtools,mathtools}
\usepackage{graphicx,enumerate,stmaryrd}
\usepackage{xparse,hyperref}
\usepackage{anyfontsize}
\SetSymbolFont{stmry}{bold}{U}{stmry}{m}{n}
\allowdisplaybreaks[4]
\usepackage[all]{xy}
\usepackage[table]{xcolor}
\usepackage[normalem]{ulem} 
\usepackage{xcolor, bm, tensor} 
\usepackage{tikz,tikz-cd}
\usetikzlibrary{matrix}
\usetikzlibrary{fit}
\usetikzlibrary{calc}

\definecolor{deepcarrotorange}{rgb}{0.91, 0.41, 0.17}

\usepackage{adjustbox}

\tikzcdset{scale cd/.style={every label/.append style={scale=#1},
    cells={nodes={scale=#1}}}}

\usepackage[bb=dsserif]{mathalpha}

\newtheorem{theoremm}{Theorem}[section]

\declaretheorem[style=plain,name=Theorem,numberlike=theoremm]{theorem}
\declaretheorem[style=plain,name=Lemma,numberlike=theoremm]{lemma}
\declaretheorem[style=plain,name=Proposition,numberlike=theoremm]{proposition}
\declaretheorem[style=plain,name=Corollary,numberlike=theoremm]{corollary}
\declaretheorem[style=plain,name=Conjecture,numbered=no]{conjecture*}

\declaretheorem[style=definition,name=Definition,numberlike=theorem]{definition}
\declaretheorem[style=remark,name=Example,numberlike=theorem]{example}
\declaretheorem[style=remark,name=Remark,numberlike=theorem]{remark}
\declaretheorem[style=remark,name=Notation,numberlike=theorem]{notation}

\font\sc=rsfs10
\newcommand{\csym}[1]{\sc\mbox{#1}\hspace{1.0pt}}

\font\scc=rsfs7
\newcommand{\ccf}[1]{\scc\mbox{#1}\hspace{0.5pt}}

\newcommand{\on}[1]{\operatorname{#1}}
\newcommand{\setj}[1]{\left\{ #1 \right\}}
\newcommand{\hcomp}{\circ_{h}}
\newcommand{\vcomp}{\circ_{v}}

\newcommand{\idem}{complete set of pairwise orthogonal, primitive idempotents}

\newcommand{\zz}{\rightleftarrows}
\newcommand{\xiso}{\xrightarrow{\sim}}

\DeclareMathAlphabet\EuRoman{U}{eur}{m}{n}
\SetMathAlphabet\EuRoman{bold}{U}{eur}{b}{n}
\newcommand{\euler}{\EuRoman}

\usepackage[foot]{amsaddr}

\begin{document}
\title[Weighted colimits of 2-representations and star algebras]{Weighted colimits of 2-representations \\ and star algebras}
\author{Mateusz Stroi{\' n}ski}
\address{Uppsala Universitet}
\email{mateusz.stroinski@math.uu.se}

\begin{abstract}
 We apply the theory of weighted bicategorical colimits to study the problem of existence and computation of such colimits of birepresentations of  finitary bicategories. 
 The main application of our results is the complete classification of simple transitive birepresentations of a bicategory studied previously by Zimmermann. The classification confirms a conjecture he has made.
\end{abstract}

\maketitle
\tableofcontents

\section{Introduction}
 Systematic study of finitary $2$-representations of finitary $2$-categories was initiated by the series of papers \cite{MM1}, \cite{MM2}, \cite{MM3}, \cite{MM4}, \cite{MM5}, \cite{MM6}, after a number of successful instances and applications of categorification, and, in particular, categorical actions, in various areas of mathematics. These include the advances in knot theory following the introduction of Khovanov homology in \cite{Kh}, the proof of Brou{\' e}'s abelian defect group conjecture for symmetric groups in \cite{CR}, the categorification of quantum groups developed in \cite{KL1}, \cite{KL2}, \cite{KL3} as well as \cite{Ro},  \cite{Os}, \cite{ENO}, and many more.
 
 One of the main purposes of $2$-representation theory is to gain a better abstract understanding of categorical actions of $2$-categories (and, by extension, also of the $2$-categories themselves), such as those in the above listed applications. Finitary $2$-categories can be viewed as a $2$-categorical counterpart to the classical notion of a finite-dimensional algebra, and from this point of view, $2$-representation theory is analogous to classical representation theory of such algebras, which aims at a better understanding of the linear actions of the algebra. 
 
 One of the first problems to consider in the classical setting is the classification of simple modules. A $2$-representation theoretic analogue of simple modules, known as simple transitive $2$-representations, was introduced in \cite{MM5}, together with an associated weak Jordan-H{\" o}lder theory. Classification of such $2$-representations has since become one of the central problems of the theory, with many complete classification results obtained, such as those in \cite{MM5} and \cite{MMMTZ2}. See also \cite{Ma} for a slightly outdated overview.
 
 A common feature of most of these results is the quasi-fiat structure of the $2$-category considered. In the theory of tensor categories, this corresponds to rigidity of the tensor category. In the former case, one requires all $1$-morphisms to have left and right adjoints; in the latter, one requires all objects to have left and right duals. Continuing the analogy with finite-dimensional algebras, a fiat $2$-category can be viewed as analogous to a finite-dimensional algebra with an involution.
 
 The most notable exception  is the main result of \cite{MMZ}, which classifies simple transitive $2$-representations of a large family of finitary $2$-categories which need not be fiat. This is done by embedding the studied $2$-category in a $2$-category with additional adjunctions and lifting $2$-representations to this bigger $2$-category.
 
 The present paper classifies the simple transitive $2$-representations of a non-fiat $2$-category $\csym{B}_{n}^{\on{str}}$ which does not have the crucial {\it cell symmetry} of the non-fiat $2$-categories considered by \cite{MMZ}. The classification for $\csym{B}_{n}^{\on{str}}$ confirms \cite[Conjecture 1]{Zi}, establishing a bijection between equivalence classes of simple transitive $2$-representations and set partitions of $\setj{1,\ldots,n}$. 
 In contrast to \cite{MMZ}, neither the construction nor the classification employs an auxiliary fiat $2$-category. 
 Instead, we use weighted bicategorical colimits of prior known, not necessarily simple transitive, $2$-representations of $\csym{B}_{n}^{\on{str}}$, to construct new simple transitive $2$-representations, and use the universal properties for the classification. The weighted colimits used can be thought of as a $2$-representation theoretic categorification of quotient modules in classical representation theory. The prerequisite facts and a description of our application of such colimits is given in Section~\ref{s3}. To our best knowledge, this approach has not been considered previously in the study of $2$-representations.
 
 It was observed in \cite{Zi} that for any simple transitive $2$-representation $\mathbf{M}$ of $\csym{B}_{n}^{\on{str}}$, there is a distinguished $2$-transformation $\Sigma: \mathbf{C} \rightarrow \mathbf{M}$, i.e. a functor intertwining the $\csym{B}_{n}^{\on{str}}$-actions, from the cell $2$-representation $\mathbf{C}$. 
 Further, \cite{Zi} shows that $\Sigma$ sends indecomposable objects to indecomposable objects, possibly identifying certain isomorphism classes of such objects. The kind of potential identification observed there cannot be captured using the more familiar categorical constructions, such as orbit categories and skew group categories. 
 Our construction uses the analogy with finite-dimensional algebras: if we instead considered a finite-dimensional algebra $A$ and an $A$-module $C$, given elements $x,y \in C$ we could universally construct a morphism $\varphi$ of modules from $C$, satisfying $\varphi(x) = \varphi(y)$, as the projection $C \twoheadrightarrow C/\langle x - y \rangle$, which is the coequalizer of the diagram $\begin{tikzcd} A \arrow[r, shift left, "1 \mapsto x"] \arrow[r, shift right, swap, "1 \mapsto y"] & C\end{tikzcd}$. It is this latter realization we mimic in the bicategorical setting, replacing the regular module by the (representable) principal $2$-representation $\mathbf{P}$ and, using Yoneda lemma, viewing indecomposable objects $X,Y$ of $\mathbf{C}$ as parallel $2$-transformations from $\mathbf{P}$ to $\mathbf{C}$. We then study the bicategorical colimit which universally renders the $2$-transformations {\it isomorphic} (rather than equal).
 
 This stands in stark contrast with the techniques employed in the theory of fiat $2$-categories, where one heavily relies on $2$-representations obtained from structures internal to the $2$-category whose $2$-representations are studied, such as (co)algebra $1$-morphisms and internal $\on{Hom}$ studied in \cite{MMMT}, whereas our approach uses the structure of the $2$-category of $2$-representations. In fact, we need to embed the latter $2$-category into the $2$-category of all $\Bbbk$-linear $2$-functors to the $2$-category $\mathbf{Cat}_{\Bbbk}$ of $\Bbbk$-linear categories, which can be viewed as analogous to the embedding of the category $A\!\on{-mod}$ of finitely generated modules over a $\Bbbk$-algebra $A$, to the category $A\!\on{-Mod}$ of all its modules.

 This new approach, which in a sense categorifies quotient modules, can prove useful in similar problems concerning construction and classification of categorical actions. Indeed, a different classification problem for $2$-representations of a non-finitary (although in a sense locally finitary) $2$-category has been solved in \cite{JS} using the above approach, confirming \cite[Conjecture 2]{Jo2} and generalizing it to the setting of \cite{Jo1}. In fact, the $2$-representations constructed in \cite{JS} and in this document are the first non-trivial non-cell $2$-representations constructed in the theory of simple transitive $2$-representations of non-fiat $2$-categories.
 
 We now briefly explain the definition of the $2$-category $\csym{B}_{n}^{\on{str}}$. 
 First, consider the double quiver $\mathtt{R}_{n}$ on the star graph $\mathtt{G}_{n}$ on $n+1$ vertices, as depicted below for $n=5$:
 \newlength{\edgelentgh}
\setlength{\edgelentgh}{1cm}
 \[
\mathtt{G}_{5} = \begin{tikzcd}[row sep={0cm,between origins},column sep={0cm,between origins}] 
&[.3090169944\edgelentgh] &[.5\edgelentgh] 2 &[.5\edgelentgh] &[.3090169944\edgelentgh] \\[.5877852523\edgelentgh] 3 &&&& 1
\\[.24510565163\edgelentgh]
&& 0 \arrow[uu, no head] \arrow[urr, no head] \arrow[ull, no head] \arrow[dl, no head] \arrow[dr, no head] \\[.70595086467\edgelentgh]
& 4 && 5
\end{tikzcd} \quad \text{ and } \quad \mathtt{R}_{5} =
\begin{tikzcd}[row sep={0cm,between origins},column sep={0cm,between origins}] 
&[.3090169944\edgelentgh] &[.5\edgelentgh] 2 &[.5\edgelentgh] &[.3090169944\edgelentgh] \\[.5877852523\edgelentgh] 3 &&&& 1
\\[.24510565163\edgelentgh]
&& 0 \arrow[uu, bend left=15] \arrow[from = uu, bend left=15] \arrow[urr, bend left=15] \arrow[from = urr, bend left=15] \arrow[ull, bend left= 15] \arrow[from=ull, bend left=15] \arrow[dl, bend left=15] \arrow[from=dl, bend left=15] \arrow[dr, bend left=15] \arrow[from = dr, bend left=15] \\[.70595086467\edgelentgh]
& 4 && 5
\end{tikzcd}.
\]     
 The zigzag algebra $A_{n}$ on $\mathtt{G}_{n}$ is a quotient of the path algebra of $\mathtt{R}_{n}$ by the ideal generated by paths $i \rightarrow j \rightarrow k$ with $i\neq k$, together with elements of the form $(i \rightarrow j \rightarrow i) - (i \rightarrow j' \rightarrow i)$, for $j \neq j'$. In particular, $A_{n}$ is weakly symmetric, hence also self-injective.
 Under the \idem\ $\setj{e_{0},e_{1},\ldots,e_{n}}$, induced by the above labelling of the quiver, we may consider the set $\mathtt{L}_{0} = \setj{A_{n}} \cup \setj{A_{n}e_{i} \otimes_{\Bbbk} e_{0}A_{n} \; | \; i = 0,1,\ldots,n}$ of $A_{n}$-$A_{n}$-bimodules. The additive, $\Bbbk$-linear category $\mathcal{B}_{n} := \on{add}\mathtt{L}_{0}$ is a monoidal subcategory of $(A_{n}\!\on{-mod-}A_{n}, \otimes_{A_{n}})$. Observe that $\mathcal{B}_{n}$ is not symmetric or braided. We denote by $\csym{B}_{n}$ the delooping bicategory of $\mathcal{B}_{n}$. The $2$-category $\csym{B}_{n}^{\on{str}}$ is a strictification of $\csym{B}_{n}$. More precisely, it is obtained by delooping the strict monoidal category $\mathcal{B}_{n}^{\on{str}}$ of right exact endofunctors of $A_{n}\!\on{-mod}$ isomorphic to those given by the objects of $\mathcal{B}_{n}$. 
 
 We follow \cite{MMMTZ1} in relaxing the $2$-categorical setup to the bicategorical setup, which allows us to consider simple transitive birepresentations of $\csym{B}_{n}$ rather than simple transitive $2$-representations of $\csym{B}_{n}^{\on{str}}$. As observed in \cite{MMMTZ1}, the resulting two classification problems are equivalent. The bicategorical setup allows for greater flexibility in the computation of colimits.
 
 We can now state our main result precisely:
  \begin{conjecture*}[{\cite[Conjecture 1]{Zi}}, {Theorem~\ref{TheMainTheorem}}]
  Equivalence classes of simple transitive $2$-repre\-senta\-tions of $\csym{B}_{n}^{\on{str}}$ are in bijection with set partitions of $\setj{1,\ldots,n}$.
 \end{conjecture*}
 
 The paper is organized as follows. Section \ref{s2} contains the necessary preliminaries for the techniques of $2$-representation theory we will employ, as well as a complete account of the notation we will use.
 In Section \ref{s3} we give an elementary account of weighted colimits, bicategorical cocompleteness of bicategories of $\Bbbk$-linear pseudofunctors and preservation of weighted colimits by additive and Karoubi envelopes. 
 Section \ref{s4} defines $\csym{B}_{n}$, summarizes and extends the results of \cite{Zi}, giving necessary properties of simple transitive birepresentations of $\csym{B}_{n}$ without proving their existence. Section \ref{s5} constructs the simple transitive birepresentations and uses the results of Section \ref{s4} to obtain the classification.
 ~\\
 
 \noindent \textbf{Acknowledgments.} This research is partially supported by G{\" o}ran Gustafssons Stiftelse. The author would like to thank his advisor, Volodymyr Mazorchuk, for numerous helpful discussions, and also for providing and explaining the proof of Lemma \ref{NilpotentAnn}. The author would also like to thank Richard Garner, and the anonymous referee, for their  valuable remarks.

\section{Preliminaries}\label{s2}

 Throughout the text we always require the structure $2$-morphisms of bicategorical structures to be invertible.
 The resulting setting of bicategories, pseudofunctors, strong transformations and modifications is what we will call the {\it bicategorical setting}, and we will give our results in this setup. In particular, we will study birepresentations of finitary bicategories, following \cite{MMMTZ1}.
 
 Most of our results also hold in what we call the {\it $\mathit{2}$-categorical setting}, where we require the structure $2$-morphisms to be the identities, thus working with $2$-categories, $2$-functors, $2$-transformations and modifications. We will comment on possible differences between the bicategorical and the $2$-categorical results, when suitable.
 
\subsection{Notation}
 
 Our notational conventions largely follow those of \cite{MMMTZ1}, with one difference and a few additions.
 \begin{itemize}
  \item Bicategories are denoted by $\csym{B},\csym{C}$, and the like. 
  Pseudofunctors are denoted by $\mathbf{M},\mathbf{N}$, and the like.
  Strong transformations are denoted by capital Greek letters, e.g. $\Sigma, \Theta$.
  Modifications are denoted by $\euler{m}, \euler{n}$ and the like.
  \item Structure $2$-morphisms of bicategorical structures are denoted by lower case fraktur. In particular, associators are denoted by $\mathfrak{a}$, left unitors are denoted by $\mathfrak{l}$, and right unitors are denoted by $\mathfrak{r}$. These $2$-morphisms are denoted by $\alpha, v^{l}, v^{r}$ in \cite{MMMTZ1}.
  \item Categories are denoted by $\mathcal{C,D}$, and the like; objects in a category $\mathcal{C}$ are denoted by capital letters, such as $X \in \on{Ob}\mathcal{C}$. Morphisms are denoted by lower case letters, such as $f \in \mathcal{C}(X,Y)$. 
  \item Given bicategories $\csym{B},\csym{C}$, we denote the bicategory of pseudofunctors from $\csym{B}$ to $\csym{C}$ by $[\csym{B},\csym{C}]$. 
  \item Objects in a bicategory are denoted by $\mathtt{i,j}$, and the like. $1$-morphisms in a bicategory are denoted by $\mathrm{F},\mathrm{G}$, and the like. $2$-morphisms are denoted by lower case Greek letters, e.g. $\alpha, \beta$. The identity $1$-morphisms of objects will be denoted by $\mathbb{1}_{\mathtt{i}}$ and the like, and the identity $2$-morphisms of $1$-morphisms will be denoted by $\on{id}_{\mathrm{F}}$, and the like.
  \item We write $\mathrm{GF} = \mathrm{G} \hcomp \mathrm{F}$ for composition of $1$-morphisms, $\beta \hcomp \alpha$ for horizontal composition of $2$-morphisms, $\beta \vcomp \alpha$ for vertical composition of $2$-morphisms. The {\it whiskering} of a $1$-morphism $\mathrm{G}$ with a $2$-morphism $\alpha$, given by $\on{id}_{\mathrm{G}} \hcomp \alpha$, will be denoted by $\mathrm{G} \bullet \alpha$. Similarly for $\alpha \bullet \mathrm{G}$.
  \item To emphasize the strictness of $2$-categorical structures, we use the superscript $(-)^{\on{str}}$. For instance, if $\csym{C},\csym{D}$ are $2$-categories, we denote the $2$-category of $2$-functors from $\csym{C}$ to $\csym{D}$ by $[\csym{C},\csym{D}]^{\on{str}}$, and, given a bicategory $\csym{C}$ and a $2$-category $\csym{D}$ biequivalent to $\csym{C}$, we may denote $\csym{D}$ by $\csym{C}^{\on{str}}$.
 \end{itemize}
 Similarly to, \cite{MMMTZ1}, we only indicate the $1$-morphisms indexing the structure $2$-morphisms of a bicategorical structure, while omitting the indexing objects from the notation, thus writing $\mathfrak{a}_{H,G,F}$ rather than $\mathfrak{a}_{H,G,F}^{\mathtt{i,j,k,l}}$ for the associator 
  \[
   (\mathrm{HG})\mathrm{F} \xiso \mathrm{H}(\mathrm{GF}) \text{ of the composition } \mathtt{i} \xrightarrow{\mathrm{F}} \mathtt{j} \xrightarrow{\mathrm{G}} \mathtt{k} \xrightarrow{\mathrm{H}} \mathtt{l}.
  \]
 Note that our notational conventions for objects, $1$-morphisms and $2$-morphisms of a bicategory do not apply to the pseudofunctor bicategory $[\csym{C},\csym{D}]$, where we prioritize our conventions for pseudofunctors, strong transformations and modifications.
 Similarly, our conventions do not apply to functor categories, or to the $2$-category $\mathbf{Cat}$, where we prioritize our separate conventions for categories and functors.
 
\subsection{Finitary bicategories and their birepresentations}

Let $\Bbbk$ be an algebraically closed field of characteristic zero. We remark that, for the statements that do not involve finitary categories, this assumption can be relaxed to the assumption that $\Bbbk$ is a commutative ring. 

Let $\mathbf{Cat}_{\Bbbk}$ denote the $2$-category of small $\Bbbk$-linear categories, $\Bbbk$-linear functors and natural transformations. Under tensor product of $\Bbbk$-linear categories, $\mathbf{Cat}_{\Bbbk}$ becomes a symmetric monoidal $2$-category. 
 For the general definition of a monoidal $2$-category, we refer to \cite[Definition 2.6]{GPS}.

We say that a bicategory $\csym{B}$ is {\it $\Bbbk$-linear} if it is enriched in $\mathbf{Cat}_{\Bbbk}$. Similarly for $\Bbbk$-linear pseudofunctors, strong transformations and modifications. For the general definition and extensive treatment of bicategories enriched in a monoidal bicategory, we refer to \cite{GS}. In the case of $\mathbf{Cat}_{\Bbbk}$, it follows that a $\Bbbk$-linear bicategory is a bicategory $\csym{B}$ such that 
for any $\mathtt{i,j} \in \on{Ob}\csym{B}$, the category $\csym{B}(\mathtt{i,j})$ is $\Bbbk$-linear and horizontal composition $\hcomp$ is $\Bbbk$-bilinear.

$\mathbf{Cat}_{\Bbbk}$ itself is a $\Bbbk$-linear $2$-category: for any $\Bbbk$-linear categories $\mathcal{C},\mathcal{D}$, the category $\mathbf{Cat}_{\Bbbk}(\mathcal{C},\mathcal{D})$ is $\Bbbk$-linear under pointwise formation of $\Bbbk$-linear combinations of natural transformations of $\Bbbk$-linear functors.

Given $\Bbbk$-linear bicategories $\csym{B}, \csym{C}$, a $\Bbbk$-linear pseudofunctor $\mathbf{M}: \csym{B} \rightarrow \csym{C}$ is a pseudofunctor of underlying ordinary bicategories, such that, for any objects $\mathtt{i,j}$ of $\csym{B}$, the local functor $\mathbf{M}_{\mathtt{i,j}}$ is $\Bbbk$-linear. 
A $\Bbbk$-linear strong transformation of $\Bbbk$-linear pseudofunctors is a strong transformation of underlying pseudofunctors, with no additional requirements. Similarly, $\Bbbk$-linear modifications are just modifications of said strong transformations. We will omit specifying $\Bbbk$-linearity of bicategorical structures whenever it is a vacuous condition.

Whenever speaking of an ambient bicategory $\csym{B}$, $2$-category $\csym{C}$ or category $\mathcal{C}$, we implicitly assume it to be essentially small. Our main aim is to prove the existence of certain $\Bbbk$-linear pseudofunctors from a fixed, essentially small, bicategory $\csym{B}$ to $\mathbf{Cat}_{\Bbbk}$. It will become clear that this result is invariant under $\Bbbk$-linear biequivalence, so we may first construct such pseudofunctors from a biequivalent, small $\csym{B}'$ and then pass under biequivalence. Thus, for our purposes, we may further assume that said essentially small structures are in fact small.

 Composition of two $\Bbbk$-linear pseudofunctors is again a $\Bbbk$-linear pseudofunctor, and so the collection of such pseudofunctors between the $\Bbbk$-linear bicategories $\csym{B}$ and $\csym{C}$, together with strong transformations and modifications, forms a bicategory which we denote by $[\csym{B},\csym{C}]_{\Bbbk}$. In accordance with our notational conventions, if $\csym{B},\csym{C}$ are $\Bbbk$-linear $2$-categories, we denote the corresponding $2$-category of $2$-functors, $2$-transformations and modifications by $[\csym{B},\csym{C}]_{\Bbbk}^{\on{str}}$.
 
  We say that a pair 
\[
(\mathbf{F,G}), \quad \mathbf{F}: \csym{C} \rightarrow \csym{D}, \mathbf{G}: \csym{D} \rightarrow \csym{C}
\]
of pseudofunctors is {\it bicategorically adjoint} if there are equivalences of categories
\[
 \csym{D}(\mathbf{F}\mathtt{c},\mathtt{d}) \simeq \csym{C}(\mathtt{c},\mathbf{G}\mathtt{d}),
\]
strongly natural in $\mathtt{c,d}$. If $\csym{C},\csym{D}$ are $2$-categories, $\mathbf{F,G}$ are $2$-functors, and there are isomorphisms of categories
\[
 \csym{D}(\mathbf{F}\mathtt{c},\mathtt{d}) \simeq \csym{C}(\mathtt{c},\mathbf{G}\mathtt{d}),
\]
$2$-natural in $\mathtt{c,d}$, we say that $(\mathbf{F,G})$ is {\it $2$-categorically adjoint}. In particular, we avoid the common terminology calling bicategorical adjunctions biadjunctions, since $1$-categorical ambidextrous adjunctions, abundant in $2$-representation theory, are often referred to as biadjoint pairs. For an extensive account of bicategorical adjunctions, see \cite{Fi}.

Following the observation \cite[2.29]{Ke1}, given a $\Bbbk$-linear category $\mathcal{A}$, the pair  
\[
(- \otimes_{\Bbbk} \mathcal{A}, \mathbf{Cat}_{\Bbbk}(\mathcal{A},-))
\]
of $2$-endofunctors of $\mathbf{Cat}_{\Bbbk}$ is $2$-categorically adjoint, and hence the symmetric monoidal $2$-category $\mathbf{Cat}_{\Bbbk}$ is {\it closed.}
As remarked in \cite[Section 5]{GS}, given $\Bbbk$-linear bicategories $\csym{B},\csym{C},\csym{D}$, we may form the $\Bbbk$-linear bicategory $\csym{B} \otimes_{\Bbbk} \csym{C}$, given by products on the level of objects and $1$-morphisms and by tensor product over $\Bbbk$ on the level of $2$-morphisms and structure $2$-morphisms, which yields the canonical $\Bbbk$-linear biequivalences
\[
 [\csym{B},[\csym{C},\csym{D}]_{\Bbbk}]_{\Bbbk} \simeq [\csym{B} \otimes_{\Bbbk} \csym{C}, \csym{D}]_{\Bbbk} \simeq [\csym{C} \otimes_{\Bbbk} \csym{B}, \csym{D}]_{\Bbbk} \simeq [\csym{C},[\csym{B},\csym{D}]_{\Bbbk}]_{\Bbbk}.
\]
We may also form the $\Bbbk$-linear bicategory $\csym{B}^{\on{op}}$, by reversing the direction of $1$-morphisms in $\csym{B}$.
 
 We say that a $\Bbbk$-linear category $\mathcal{C}$ is {\em finitary} if it is equivalent to the category $A\!\on{-proj}$ of finite dimensional projective modules over a finite dimensional associative $\Bbbk$-algebra $A$. 
 
 \begin{definition}
 A {\it finitary bicategory} is a bicategory $\csym{B}$ which is $\Bbbk$-linear and such that the category $\csym{B}(\mathtt{i,j})$ is finitary, for all $\mathtt{i,j} \in \on{Ob}\csym{B}$. 
 For the remainder of this section, let $\csym{B}$ be a finitary bicategory.
 \end{definition}

 Let $\mathfrak{A}_{\Bbbk}^{f}$ denote the $2$-category of finitary categories, $\Bbbk$-linear functors and natural transformations. A {\it finitary birepresentation} of $\csym{B}$ is a $\Bbbk$-linear pseudofunctor from $\csym{B}$ to $\mathfrak{A}_{\Bbbk}^{f}$. We denote the bicategory of finitary birepresentations of $\csym{B}$ by $\csym{B}\!\on{-afmod}$.
 
 Given a finitary category $\mathcal{A}$, the split Grothendieck group $[\mathcal{A}]_{\oplus}$ of $\mathcal{A}$ is a free abelian group of finite rank. Let $F: \mathcal{A} \rightarrow \mathcal{B}$ be a $\Bbbk$-linear functor between finitary categories. Under a choice of bases for $[\mathcal{A}]_{\oplus}$ and $[\mathcal{B}]_{\oplus}$, the induced group homomorphism $[F]_{\oplus}$ is represented by a matrix. 
 Given a finitary birepresentation $\mathbf{M}$ of $\csym{B}$, together with choices of bases for $[\mathbf{M}(\mathtt{i})]$ as $\mathtt{i} \in \csym{B}$, we obtain a matrix $[\mathbf{M}\mathrm{F}]_{\oplus}$ for every $1$-morphism $\mathrm{F}$ of $\csym{B}$ - the so-called {\it action matrix} of $\mathrm{F}$ under $\mathbf{M}$.
 
 A $\Bbbk$-linear, abelian category $\mathcal{C}$ is said to be {\it finite} if it is equivalent to the category $A\!\on{-mod}$ of finite dimensional modules over a finite dimensional associative $\Bbbk$-algebra $A$. We let $\mathfrak{R}_{\Bbbk}$ denote the $2$-category of $\Bbbk$-linear finite abelian categories, right exact functors and natural transformations. An {\it abelian birepresentation} of $\csym{B}$ is a $\Bbbk$-linear pseudofunctor from $\csym{B}$ to $\mathfrak{R}_{\Bbbk}$. Given a finitary birepresentation $\mathbf{M}$ of $\csym{B}$, its {\it projective abelianization} $\overline{\mathbf{M}}$ is an abelian birepresentation such that $\mathbf{M}$ can be recovered from $\overline{\mathbf{M}}$ by restricting to certain subcategories equivalent to $\overline{\mathbf{M}}(\mathtt{i})\!\on{-proj}$, for $\mathtt{i} \in \on{Ob}\csym{B}$. We refer the reader to \cite{MM2} for details, and \cite{MMMT} for an improved construction. We will only use abelianization once in this document, and in that case the simpler construction of \cite{MM2} can be used.
 
 We say that a birepresentation of $\csym{B}$ is {\it transitive} if, for any $\mathtt{i,j} \in \on{Ob}\csym{B}$ and any $X \in \on{Ob}\mathbf{M}(\mathtt{i}), Y \in \on{Ob}\mathbf{M}(\mathtt{j})$, there is a $1$-morphism $\mathrm{F} \in \csym{B}(\mathtt{i,j})$ such that $Y$ is isomorphic to a direct summand of $\mathbf{M}\mathrm{F}(X)$. 
 
 A {\it $\csym{B}$-stable ideal} $\mathbf{I}$ of a finitary birepresentation $\mathbf{M}$ of $\csym{B}$ is a tuple $\big(\mathbf{I}(\mathtt{i})\big)_{\mathtt{i} \in \on{Ob}\ccf{B}}$ of ideals of $\mathbf{M}(\mathtt{i})$ such that, for any $X\xrightarrow{f} Y \in \mathbf{I}(\mathtt{i})$ and any $\mathrm{F} \in \on{Ob}\csym{B}(\mathtt{i,j})$, we have $\mathbf{M}\mathrm{F}(f) \in \mathbf{I}(\mathtt{j})$. 
 We say that $\mathbf{I}$ is {\it proper} if there is $\mathtt{i}$ such that $\setj{0} \subsetneq \mathbf{I}(\mathtt{i}) \subsetneq \mathbf{M}(\mathtt{i})$.
 For any $\alpha \in \csym{B}(\mathtt{i,j})(\mathrm{F},\mathrm{F}')$, we may define the {\it evaluation of $\alpha$ at $f$} as
 \[
  (\mathbf{M}\alpha)_{Y}\circ \mathbf{M}\mathrm{F}(f) = \mathbf{M}\mathrm{F}'(f) \circ (\mathbf{M}\alpha)_{X} \in \mathbf{I}(\mathtt{j}).
 \]
 
 \begin{lemma}\label{Preimages}
  Let $\mathbf{M} \xrightarrow{\Theta} \mathbf{N}$ be a strong transformation of finitary birepresentations of a finitary bicategory $\csym{B}$. Let $\mathbf{I}$ be a $\csym{B}$-stable ideal of $\mathbf{N}$. Given $\mathtt{i} \in \on{Ob}\csym{B}$ and $X,Y \in \mathbf{M}(\mathtt{i})$, let
  \[
   (\Theta^{-1}\mathbf{I})(\mathtt{i})(X,Y) := \setj{X \xrightarrow{f} Y \; | \; \Theta_{\mathtt{i}}(f) \in \mathbf{I}}.
  \]
 Then the assignment $\Theta^{-1}\mathbf{I} = \left( \Theta^{-1}\mathbf{I}(\mathtt{i})\right)_{\mathtt{i} \in \on{Ob}\csym{B}}$ gives a $\csym{B}$-stable ideal of $\mathbf{M}$.
 \end{lemma}

 \begin{proof}
  For every $\mathtt{i}$, the component $\Theta_{\mathtt{i}}$ is a $\Bbbk$-linear functor, and thus $(\Theta^{-1}\mathbf{I})(\mathtt{i})$ is an ideal of $\mathbf{M}(\mathtt{i})$.
  To see that $\Theta^{-1}\mathbf{I}$ is $\csym{B}$-stable, let $\mathrm{F} \in \csym{B}(\mathtt{i,j})$, and let $X\xrightarrow{f} Y$ be a morphism of $(\Theta^{-1}\mathbf{I})(\mathtt{i})$. We then have the commutative square
  \[
   \begin{tikzcd}
    \mathbf{N}\mathrm{F}\Theta_{\mathtt{i}}(X) \arrow[d, "(\Theta_{\mathrm{F}})_{X}"] \arrow[r, "\mathbf{N}\mathrm{F}\Theta_{\mathtt{i}}(f)"] & \mathbf{N}\mathrm{F}\Theta_{\mathtt{i}}(X) \arrow[d, "(\Theta_{\mathrm{F}})_{Y}"] \\
    \Theta_{\mathtt{j}}\mathbf{M}\mathrm{F}(X) \arrow[r, "\Theta_{\mathtt{j}}\mathbf{M}\mathrm{F}(f)"] & \Theta_{\mathtt{j}}\mathbf{M}\mathrm{F}(Y) 
   \end{tikzcd}
  \]
 where the upper horizontal arrow lies in $\mathbf{I}(\mathtt{j})$ by assumption, and hence the lower horizontal arrow also lies in $\mathbf{I}(\mathtt{j})$. This proves that $\mathbf{M}\mathrm{F}(f) \in (\Theta^{-1}\mathbf{I})(\mathtt{j})$. 
 \end{proof}
 
 \begin{definition}
 A finitary birepresentation $\mathbf{M}$ of $\csym{B}$ is {\it simple transitive} if it has no proper $\csym{B}$-stable ideals. In particular, a simple transitive birepresentation is transitive.
 \end{definition}
 
 \subsection{Cells}
 
 The {\em left preorder} $\leq_{L}$ on the set of isomorphism classes of indecomposable $1$-morphisms of $\csym{B}$ is defined by writing $\mathrm{F} \leq_{L} \mathrm{G}$ if there is a $1$-morphism $\mathrm{H}$ such that $\mathrm{G}$ is a direct summand of $\mathrm{HF}$. We denote the resulting equivalence relation by $\sim_{L}$, and refer to the equivalence classes as {\em left cells.} Similarly one defines the right and two-sided preorders $\leq_{R}, \leq_{J}$, together with the right and two-sided equivalence relations and right and two-sided cells.
 
 Let $\mathcal{L}$ be a left cell of $\csym{B}$. There is then a unique object $\mathtt{i}$ of $\csym{B}$ which is the domain of all $1$-morphisms in $\mathcal{L}$. To $\mathcal{L}$ we associate a simple transitive subquotient $\mathbf{C}_{\mathcal{L}}$ of the principal $2$-representation $\mathbf{P}_{\mathtt{i}} := \csym{B}(\mathtt{i},-)$; see \cite[Subsection 3.3]{MM5} for details.
 
 By \cite[Lemma 1]{CM}, given a transitive birepresentation $\mathbf{M}$ of $\csym{B}$, the set of $J$-cells of $\csym{B}$ not annihilating $\mathbf{M}$ admits a $J$-greatest element, which we call the {\it apex} of $\mathbf{M}$. The apex is an invariant of transitive birepresentations.
 
 \subsection{Biideals}
  {\it A two-sided biideal} $\mathcal{I}$ in a $\Bbbk$-linear bicategory $\csym{C}$ consists of a collection of ideals $(\mathcal{I}_{\mathtt{j,k}})_{\mathtt{j,k} \in \on{Ob}\ccf{C}}$ of $\csym{C}(\mathtt{j,k})$, such that, for any $1$-morphisms $\mathrm{F}$ of $\csym{C}(\mathtt{i,j})$ and $\mathrm{G}$ of $\csym{C}(\mathtt{k,l})$, and any $2$-morphism $\gamma \in \mathcal{I}_{\mathtt{j,k}}$, we have
 \[
  \mathrm{G} \bullet \gamma \in \mathcal{I}_{\mathtt{j,l}}, \; \gamma \bullet \mathrm{F} \in \mathcal{I}_{\mathtt{i,k}}.
 \]
  As a consequence, given any $2$-morphisms $\alpha$ of $\csym{C}(\mathtt{i,j})$ and $\beta$ of $\csym{C}(\mathtt{k,l})$, we have
  \[
   \beta \hcomp \gamma \in \mathcal{I}_{\mathtt{j,l}}, \; \gamma \hcomp \alpha \in \mathcal{I}_{\mathtt{i,k}}
  \]
 due to
 \[
   \beta \hcomp \gamma = (\beta \hcomp \on{id}_{\on{cod}\gamma}) \vcomp (\on{id}_{\on{dom}\beta} \hcomp \gamma)
 \]
 and similarly for $\alpha$. Given a biideal $\mathcal{I}$, the assignment $(\mathcal{I}^{2})_{\mathtt{k,l}} := (\mathcal{I}_{\mathtt{k,l}})^{2}$ defines a biideal: given $\beta, \alpha \in \mathcal{I}_{\mathtt{i,j}}$ such that $\beta \vcomp \alpha$ is defined, and $\gamma$ in $\csym{C}(\mathtt{j,k})$, we have
 \[
  \gamma \hcomp (\beta \vcomp \alpha) = (\gamma \hcomp \beta) \vcomp (\gamma \hcomp \alpha).
 \]
 But $\gamma \hcomp \beta, \gamma \hcomp \alpha \in \mathcal{I}_{\mathtt{i,k}}$, and so $\gamma \hcomp (\beta \vcomp \alpha) \in \mathcal{I}_{\mathtt{i,k}}^{2}$ by definition. In particular,
\begin{equation}\label{Ih2}
 \mathcal{I}_{\mathtt{i,k}}^{h,2} := \Bbbk \setj{ \alpha \in \mathcal{I}_{\mathtt{i,k}} \; | \; \alpha = \gamma \hcomp \beta, \text{ for } \mathtt{j} \in \on{Ob}\csym{C}, \; \beta \in \mathcal{I}_{\mathtt{i,j}}, \gamma \in \mathcal{I}_{\mathtt{j,k}}}
 \subseteq \mathcal{I}_{\mathtt{i,k}}^{2}.
\end{equation}
 Inductively we may define $\mathcal{I}^{m}$, for any $m \in \mathbb{Z}_{> 0}$. We say that $\mathcal{I}$ is {\it nilpotent} if there is $m$ such that $\mathcal{I}^{m}_{\mathtt{i,j}} = 0$, for all $\mathtt{i,j} \in \csym{C}$.
 
 Assume now that $\csym{C}$ is finitary, let $\mathcal{I}$ be a two-sided biideal of $\csym{C}$, let $\mathbf{M}$ be a finitary birepresentation of $\csym{C}$ and let $\mathfrak{f}_{\mathrm{H},\mathrm{G}}: \mathbf{M}\mathrm{H}\mathbf{M}\mathrm{G} \xiso \mathbf{M}(\mathrm{HG})$ be the structure $2$-morphisms of $\mathbf{M}$. The ideal $\on{ev}_{\mathbf{M}}(\mathcal{I})$ of $\mathbf{M}$ is defined by letting $\on{ev}_{\mathbf{M}}(\mathcal{I})(\mathtt{i})$ consist of $\Bbbk$-linear combinations of morphisms of the form
 \[
  \begin{tikzcd}
   W \arrow[r, "f"] & \mathbf{M}\mathrm{F}(X) \arrow[r, "\mathbf{M}\mathrm{F}(g)"] \arrow[d, "(\mathbf{M}\alpha)_{X}"] & \mathbf{M}\mathrm{F}(Y) \arrow[d, "(\mathbf{M}\alpha)_{Y}"] \\
   & \mathbf{M}\mathrm{G}(X) \arrow[r, "\mathbf{M}\mathrm{G}(g)"] & \mathbf{M}\mathrm{G}(Y) \arrow[r, "h"] & Z
  \end{tikzcd}
 \]
where $f,h$ are arbitrary morphisms of $\mathbf{M}(\mathtt{i})$, whereas the $1$-morphisms $\mathrm{F,G}$, the $2$-morphism $\alpha$ and the morphism $g$ are only required to be such that all the compositions and evaluations in the diagram are well-defined. It is clear that this collection is stable under vertical composition, since $f,h$ are arbitrary. To see that it is also closed under horizontal composition, let $\mathrm{H}$ be such that the compositions $\mathrm{HG},\mathrm{HF}$ are defined. 

Applying $\mathbf{M}\mathrm{H}$ on the defining diagram above, using the associators, we obtain:
\[
 \begin{tikzcd}[scale cd = 0.9, sep = 3.7ex]
  \mathbf{M}\mathrm{H}(W) \arrow[r,color =red, "\mathbf{M}\mathrm{H}(f)"] & \mathbf{M}\mathrm{H}\mathbf{M}\mathrm{F}(X) \arrow[rdd, color=red, "(\mathfrak{f}_{\mathrm{H,F}})_{X}"] \arrow[rrr, "\mathbf{M}\mathrm{H}\mathbf{M}\mathrm{F}(g)"] \arrow[ddd, swap, "(\mathbf{M}\mathrm{H}\bullet\mathbf{M}\alpha)_{X}"] & & & \mathbf{M}\mathrm{H}\mathbf{M}\mathrm{F}(Y) \arrow[ddd, near start, swap, "(\mathbf{M}\mathrm{H}\bullet \mathbf{M}\alpha)_{Y}"] \arrow[rdd, "(\mathfrak{f}_{\mathrm{H,F}})_{Y}"] \\
  \\
  && \mathbf{M}(\mathrm{HF})(X) & & & \mathbf{M}(\mathrm{HF})(Y) \arrow[from=lll, color = red, crossing over, near start, "\mathbf{M}(\mathrm{HF})(g)"] \arrow[ddd, color = red, crossing over, near end, "(\mathbf{M}(\mathrm{H}\bullet \alpha))_{Y}"] \\
  &\mathbf{M}\mathrm{H}\mathbf{M}\mathrm{G}(X) \arrow[rrr, near end, "\mathbf{M}\mathrm{H}\mathbf{M}\mathrm{G}(g)"] \arrow[rdd, near start, swap, "(\mathfrak{f}_{\mathrm{H,G}})_{X}"] & & & \mathbf{M}\mathrm{H}\mathbf{M}\mathrm{G}(Y) \arrow[from = rdd, color = red, near end, "(\mathfrak{f}_{\mathrm{H,G}}^{-1})_{Y}"] \arrow[rr, color = red, "\mathbf{M}\mathrm{H}(h)"] & &[-3.7em] \mathbf{M}\mathrm{H}(Z)  \\ \\
  && \mathbf{M}(\mathrm{HG})(X) \arrow[from=uuu, color = red, crossing over, near end, "(\mathbf{M}(\mathrm{H}\bullet \alpha))_{X}"] \arrow[rrr, color = red, "\mathbf{M}(\mathrm{HG})(g)"] & & & \mathbf{M}(\mathrm{HG})(Y)
 \end{tikzcd}
 \]
 The cube in the diagram commutes: its top and bottom faces commute by the naturality of $\mathfrak{f}_{\mathrm{H,F}}$, left and right by the naturality of $\mathfrak{f}_{\mathrm{H},-}$, front by the naturality of $\mathbf{M}(\mathrm{H}\bullet \alpha)$, and back by the naturality of $\mathbf{M}\mathrm{H}\bullet \mathbf{M}\alpha$. The red part of the diagram indicates that the resulting morphism still belongs to $\on{ev}_{\mathbf{M}}(\mathcal{I})$. Further, it shows that, if we modify the definition of $\on{ev}_{\mathbf{M}}(\mathcal{I})$, so that instead of letting $\alpha$ be arbitrary, we require $\alpha  = \gamma \hcomp \beta$ as in \eqref{Ih2}, we again obtain an ideal of $\mathbf{M}$, which we denote by $\on{ev}_{\mathbf{M}}(\mathcal{I}^{h,2})$, in view of the notation in \eqref{Ih2}. Inductively, we may define further ideals $\on{ev}_{\mathbf{M}}(\mathcal{I}^{h,k})$, for $k \in \mathbb{Z}_{> 0}$. The above diagram also proves that $\on{ev}_{\mathbf{M}}(\mathcal{I}^{h,2})$ is the ideal generated by morphisms of the form $\mathbf{M}\mathrm{H}(f)$, for $f \in \on{ev}_{\mathbf{M}}(\mathcal{I})$. If we express this by the suggestive notation
 \[
  \on{ev}_{\mathbf{M}}(\mathcal{I}^{h,2}) = \csym{C}\on{ev}_{\mathbf{M}}(\mathcal{I}),
 \]
 then, setting $\mathcal{I}^{h,1} := \mathcal{I}$, inductively we also have
 \begin{equation}\label{hinduction}
  \on{ev}_{\mathbf{M}}(\mathcal{I}^{h,k+1}) = \csym{C}\on{ev}_{\mathbf{M}}(\mathcal{I}^{h,k}) \text{, for } k \geq 1.
 \end{equation}
 \begin{lemma}\label{NilpotentAnn}
  If $\mathbf{M}$ is a simple transitive birepresentation of $\csym{C}$ and $\mathcal{I}$ is a nilpotent ideal of $\csym{C}$, then $\on{ev}_{\mathbf{M}}(\mathcal{I}) = 0$. In other words, $\mathcal{I}$ annihilates $\mathbf{M}$.
 \end{lemma}

 \begin{proof}
  Since $\mathbf{M}$ is simple transitive, we have $\on{ev}_{\mathbf{M}}(\mathcal{I}) = 0$ or $\on{ev}_{\mathbf{M}}(\mathcal{I}) = \mathbf{M}$. Using \eqref{hinduction} and the notation used in the paragraphs preceding it, the latter case would imply
  \[
   \on{ev}_{\mathbf{M}}(\mathcal{I}^{h,2}) = \csym{C}\on{ev}_{\mathbf{M}}(\mathcal{I}) = \csym{C}\mathbf{M} = \mathbf{M}
  \]
 and, inductively
 \[
  \on{ev}_{\mathbf{M}}(\mathcal{I}^{h,k}) = \mathbf{M}, \text{ for all } k \geq 1.
 \]
 Let $m$ be such that $\mathcal{I}^{m} = 0$.  Similarly to \eqref{Ih2}, $\mathcal{I}^{m} = 0$ implies $\mathcal{I}^{h,m} = 0$. This yields
 \[
  \on{ev}_{\mathbf{M}}(\mathcal{I}^{h,m}) = 0,
 \]
 which contradicts the earlier conclusion, thus showing that $\on{ev}_{\mathbf{M}}(\mathcal{I}) = 0$.
 \end{proof}

 \begin{lemma}\label{RadNilp}
  Let $\mathcal{I}$ be a biideal of $\csym{C}$ not containing any identity $2$-morphisms. Then $\mathcal{I}$ is nilpotent.
 \end{lemma}
 
 \begin{proof}
  Since $\csym{C}(\mathtt{i,j})$ is finitary and $\mathcal{I}_{\mathtt{i,j}}$ contains no identity $2$-morphisms, we have 
  \[
   \mathcal{I}_{\mathtt{i,j}} \subseteq \on{Rad}\csym{C}(\mathtt{i,j}) \text{, for all } \mathtt{i,j}.
  \]
  Given $\mathtt{i,j} \in \on{Ob}\csym{C}$, let $n_{\mathtt{i,j}}$ be the nilpotency degree of $\on{Rad}\csym{C}(\mathtt{i,j})$. Since $\on{Ob}\csym{C}$ is finite, we may let $n := \max\setj{n_{\mathtt{i,j}} \; | \; \mathtt{i,j} \in \on{Ob}\csym{C}}$. Clearly then $\mathcal{I}_{\mathtt{i,j}}^{n} = 0$, for all $\mathtt{i,j} \in \on{Ob}\csym{C}$, and hence $\mathcal{I}^{n} = 0$.
 \end{proof}

\section{Weighted colimits of 2-representations}\label{s3}

\subsection{Weighted colimits of \texorpdfstring{$\Bbbk$-linear}{k-linear} pseudofunctors}
We first recall the notion of a bicategorical weighted colimit. The simplified terminology we use here replaces the more precise terminology of pseudo-, bi- and lax colimits, which we will not need since we only use elementary bicategorical notions. Our choice of terminology here is the same as that in \cite{Br}, and it is consistent with the notational conventions we made before.
\begin{definition}\label{WeightedColimDef}
 Let $\csym{B}$ be a bicategory. Let $\csym{J}$ be a small bicategory and let $\mathbf{W}$ be a pseudofunctor from $\csym{J}^{\on{op}}$ to $\mathbf{Cat}$.
 Given a pseudofunctor $\mathbf{F}$ from $\csym{J}$ to $\csym{B}$, a {\it weighted bicategorical colimit} $\mathbf{W} \star \mathbf{F}$ is an object of $\csym{B}$ together with a representation of the pseudofunctor
 \[
  [\csym{J}^{\on{op}},\mathbf{Cat}]\left(\mathbf{W},\csym{B}\left(\mathbf{F}-,\mathtt{i}\right)\right)
 \]
 in $\mathtt{i}$. In other words, there are equivalences of categories
 \[
  \csym{B}\left(\mathbf{W}\star \mathbf{F},\mathtt{i}\right) \simeq [\csym{J}^{\on{op}},\mathbf{Cat}]\left(\mathbf{W},\csym{B}\left(\mathbf{F}-,\mathtt{i}\right)\right)
 \]
strongly natural in $\mathtt{i}$. If $\mathbf{W}\star \mathbf{F}$ exists, it is unique up to a compatible internal equivalence in $\csym{B}$. If $\mathbf{W} \star \mathbf{F}$ exists for all choices of $\csym{J},\mathbf{W},\mathbf{F}$, we say that $\csym{B}$ is {\it bicategorically cocomplete.}

If we instead have a $2$-category $\csym{C}$, a small $2$-category $\csym{J}$ and $2$-functors $\mathbf{W},\mathbf{F}$, the {\it weighted $2$-categorical colimit} $\mathbf{W} \star \mathbf{F}$ is an object of $\csym{C}$ representing the $2$-functor
 \[
  [\csym{J}^{\on{op}},\mathbf{Cat}]^{\on{str}}\left(\mathbf{W},\csym{B}\left(\mathbf{F}-,\mathtt{i}\right)\right),
 \]
giving rise to isomorphisms of categories
 \[
  \csym{C}\left(\mathbf{W}\star \mathbf{F},\mathtt{i}\right) \simeq [\csym{J}^{\on{op}},\mathbf{Cat}]^{\on{str}}\left(\mathbf{W},\csym{C}\left(\mathbf{F}-,\mathtt{i}\right)\right)
 \]
 $2$-natural in $\mathtt{i}$. It is then unique up to isomorphism in $\csym{C}$, and if it always exists, we say that $\csym{C}$ is {\it $2$-categorically cocomplete.}
 
 Bicategorical and $2$-categorical limits are obtained as bicategorical and $2$-categorical colimits in $\csym{B}^{\on{op}}$ and $\csym{C}^{\on{op}}$.
\end{definition}

 \begin{proposition}\label{Cocomplete}
 The $2$-category $\mathbf{Cat}_{\Bbbk}$ is complete and cocomplete, both bicategorically and $2$-categorically.
 \end{proposition}
 \begin{proof}
 Cocompleteness can be shown by explicitly constructing certain colimits which can be used to obtain all colimits. A proof given by construction of coproducts, coinserters and coequifiers can be found in \cite[Proposition 2.6]{Br}. 
 For completeness, one easily verifies that the explicit constructions of products, cotensors, $2$-equalizers and pseudoequalizers in $\mathbf{Cat}$ given in \cite{Ca} apply also in the case of $\mathbf{Cat}_{\Bbbk}$ - this can be viewed as a consequence of preservation of limits by the forgetful $2$-functor $\mathbf{Cat}_{\Bbbk} \rightarrow \mathbf{Cat}$, which is $2$-categorically right adjoint to the free $\Bbbk$-linear category $2$-functor. Completeness follows as a consequence of \cite[1.24]{St}.
 \end{proof}

\begin{definition}
 Let $\csym{B}$ be a $\Bbbk$-linear bicategory. Let $\csym{J}$ be a small $\Bbbk$-linear bicategory and let $\mathbf{W}$ be a $\Bbbk$-linear pseudofunctor from $\csym{J}^{\on{op}}$ to $\mathbf{Cat}_{\Bbbk}$. Given a $\Bbbk$-linear pseudofunctor $\mathbf{F}: \csym{J} \rightarrow \csym{B}$, a {\it weighted $\Bbbk$-linear bicategorical colimit $\mathbf{W}\star \mathbf{F}$} is an object of $\csym{B}$ together with a representation of the $\Bbbk$-linear pseudofunctor
 \[
  [\csym{J}^{\on{op}},\mathbf{Cat}_{\Bbbk}]_{\Bbbk}(\mathbf{W},\csym{B}(\mathbf{F}-,\mathtt{i}))
 \]
 in $\mathtt{i}$. In other words, there are $\Bbbk$-linear equivalences of categories
 \[
  \csym{B}(\mathbf{W}\star \mathbf{F}, \mathtt{i}) \simeq [\csym{J}^{\on{op}},\mathbf{Cat}_{\Bbbk}]_{\Bbbk}(\mathbf{W},\csym{B}(\mathbf{F}-,\mathtt{i})).
 \]
 Similarly one obtains the notion of a $2$-categorical $\Bbbk$-linear colimit, of $\Bbbk$-linear bicategorical and $2$-categorical limits and the resulting notions of $\Bbbk$-linear bicategorical and $2$-categorical cocompleteness and completeness. 
\end{definition}

The next statement can be viewed as a direct corollary of \cite[Proposition 11.2]{GS}, however, since that paper is written in a much more general setting, we give an explanation of how one views our particular case as an instance of the setting of \cite{GS}, and how one uses the results therein to obtain the statement.

\begin{proposition}
 Let $\csym{B}$ be a small $\Bbbk$-linear bicategory. The $2$-category $[\csym{B},\mathbf{Cat}_{\Bbbk}]_{\Bbbk}$ of $\Bbbk$-linear pseudofunctors from $\csym{B}$ to $\mathbf{Cat}_{\Bbbk}$ is $\Bbbk$-linear cocomplete, and the colimits are computed pointwise. In other words, given a diagram $\mathbf{F}: \csym{J} \rightarrow [\csym{B},\mathbf{Cat}_{\Bbbk}]_{\Bbbk}$, a weight $\mathbf{W}$ in $[\csym{J},\mathbf{Cat}_{\Bbbk}]_{\Bbbk}$, and an object $\mathtt{b} \in \csym{B}$, the $\Bbbk$-linear category $(\mathbf{W}\star \mathbf{F})(\mathtt{b})$ is the colimit $\mathbf{W}\star \mathbf{F}_{\mathtt{b}}$, where $\mathbf{F}_{\mathtt{b}}$ is obtained by evaluating at $\mathtt{b}$ the $\Bbbk$-linear pseudofunctor $\mathbf{F}^{\zz}$ obtained by
 \[
  \begin{aligned}
   [\csym{J},[\csym{B},\mathbf{Cat}_{\Bbbk}]_{\Bbbk}]]_{\Bbbk} &\xiso [\csym{B},[\csym{J},\mathbf{Cat}_{\Bbbk}]_{\Bbbk}]]_{\Bbbk} \\
   \mathbf{F} &\mapsto \mathbf{F}^{\zz}
  \end{aligned}
 \]
\end{proposition}

\begin{proof}
 As we have observed before, we consider the case of enriching monoidal bicategory $\mathcal{V} = \mathbf{Cat}_{\Bbbk}$. As remarked in \cite[Section 5]{GS}, it is easy to verify that, given $\Bbbk$-linear bicategories $\csym{A}$, $\csym{C}$, left $\csym{A}$-modules in the sense of \cite{GS} are $\Bbbk$-linear pseudofunctors from $\csym{A}$ to $\mathbf{Cat}_{\Bbbk}$, and similarly, $\csym{A}$-$\csym{C}$-bimodules are $\Bbbk$-linear pseudofunctors from $\csym{A} \otimes_{\Bbbk} \csym{C}^{\on{op}}$ to $\mathbf{Cat}_{\Bbbk}$. Given $\mathbf{M},\mathbf{N} \in [\csym{B}, \mathbf{Cat}_{\Bbbk}]_{\Bbbk}$, their internal hom, as studied in the general case in \cite[Section 7]{GS}, is given by the $\Bbbk$-linear category $[\csym{B},\mathbf{Cat}_{\Bbbk}]_{\Bbbk}(\mathbf{M},\mathbf{N})$, under the clear choice of the evaluation morphism $\xi$ of \cite[Section 7.3]{GS}. Combining this with the fact that $\csym{B}$ is assumed to be small, we conclude that the $\Bbbk$-linear bicategory of moderate right $\csym{B}$-modules studied in \cite[Section 11]{GS} coincides with $[\csym{B}^{\on{op}},\mathbf{Cat}_{\Bbbk}]_{\Bbbk}$. Its $\Bbbk$-linear cocompleteness follows from \cite[Proposition 11.2]{GS}. The colimit $\mathbf{W} \star \mathbf{F}$ is given by $\mathbf{W} \otimes_{\!\ccf{J}} \tilde{\mathbf{F}}$, for $\tilde{\mathbf{F}}$ and $\otimes_{\!\ccf{J}}$ as defined in \cite{GS}. Similarly, the colimit $\mathbf{W} \star \mathbf{F}_{\mathtt{b}}$ is given by $\mathbf{W} \otimes_{\!\ccf{J}} \tilde{\mathbf{F}}_{\mathtt{b}}$. Yoneda Lemma implies that $\tilde{\mathbf{F}}(\mathtt{b}) \simeq \tilde{\mathbf{F}}_\mathtt{b}$. Taking all this into account, one sees that the result follows from \cite[Corollary 6.10]{GS}.
\end{proof}

Using the categorification of \cite[3.8]{Ke1} in \cite[13.14]{GS}, we conclude that ordinary bicategorical colimits in a $\Bbbk$-linear bicategory $\csym{B}$ can be viewed as special cases of $\Bbbk$-linear bicategorical colimits. We remark that in our case, the monoidal bicategorical adjunction obtained in \cite[13.14]{GS} is that given by the free $\Bbbk$-linear category $2$-functor $\mathbf{Cat} \rightarrow \mathbf{Cat}_{\Bbbk}$ and the forgetful $2$-functor. As a consequence of the last two observations, together with Proposition \ref{Cocomplete}, we find the following:
\begin{proposition}\label{PseudoFCoco}
 Let $\csym{B}$ be a $\Bbbk$-linear bicategory. The $\Bbbk$-linear pseudofunctor bicategory $[\csym{B},\mathbf{Cat}_{\Bbbk}]_{\Bbbk}$ is bicategorically complete and cocomplete, and the bicategorical limits and colimits are constructed pointwise from those in $\mathbf{Cat}_{\Bbbk}$.
\end{proposition}

For much more general results regarding weighted bicategorical limits, in a much more general setting of bicategories enriched in monoidal bicategories which are not necessarily symmetric or closed, see \cite[Section 10, Section 11]{GS}.

\subsection{Additive and Karoubi envelopes}

Recall that a preadditive category is a category enriched in the category $\mathbf{Ab}$ of abelian groups. Let $\mathbf{Cat}_{\mathbb{Z}}$ denote the $2$-category of preadditive categories.

The additive envelope is the universal solution to the problem of adding direct sums to a preadditive category. Similarly, the Karoubi envelope universally makes a category idempotent split. These constructions are used, for instance, to define the category $\on{Rep}(S_{t})$ - see \cite{De}. A more detailed account is given in \cite{CW}, and a very detailed account, which will be our main reference, is given in \cite{Ri}.

Let $\mathbf{Cat}_{\mathbb{Z}}^{\oplus},\mathbf{Cat}_{\mathbb{Z}}^{\euler{K}}$ denote the $1,2$-full $2$-subcategories of $\mathbf{Cat}_{\mathbb{Z}}$, with the respective object sets being that of categories with finite direct sums and that of idempotent split categories, respectively. Let $\mathbf{Cat}_{\Bbbk}^{\oplus}, \mathbf{Cat}_{\Bbbk}^{\euler{K}}$ be the $\Bbbk$-linear variants thereof. Let $\mathbf{Cat}_{\Bbbk}^{\euler{D}}$ denote the $1,2$-full $2$-subcategory of $\mathbf{Cat}_{\Bbbk}$ given by the categories which are both additive and idempotent split. We have the following results:
\begin{itemize}
 \item \cite[2.1.6, 2.2.3]{Ri}: There are $2$-functors 
 \[
(-)^{\oplus}: \mathbf{Cat}_{\mathbb{Z}} \rightarrow \mathbf{Cat}^{\oplus}_{\mathbb{Z}} \text{ and }
 (-)^{\euler{K}}: \mathbf{Cat}_{\mathbb{Z}} \rightarrow \mathbf{Cat}_{\mathbb{Z}}^{\euler{K}}.
 \]
 \item \cite[2.1.9, 2.2.6]{Ri}: The above $2$-functors are compatible with $\Bbbk$-linear structures: restriction gives $\Bbbk$-linear $2$-functors 
 \[
  (-)^{\oplus}: \mathbf{Cat}_{\Bbbk} \rightarrow \mathbf{Cat}_{\Bbbk}^{\oplus} \text{ and } (-)^{\euler{K}}: \mathbf{Cat}_{\Bbbk} \rightarrow \mathbf{Cat}_{\Bbbk}^{\euler{K}}.
 \]
 \item \cite[Proposition 41]{Ri}: The Karoubi envelope of an additive category is additive. We thus obtain the $2$-functor $(-)^{\euler{D}} := (-)^{\euler{K}} \circ (-)^{\oplus} \in [\mathbf{Cat}_{\Bbbk}, \mathbf{Cat}_{\Bbbk}^{\euler{D}}]^{\on{str}}$.
 \item \cite[2.1.3, 2.2.2]{Ri}: Given $\mathcal{C} \in \on{Ob}\mathbf{Cat}_{\Bbbk}$, there are canonical fully faithful functors $\Phi_{\mathcal{C}}: \mathcal{C} \rightarrow \mathcal{C}^{\oplus}$ and $\Psi_{\mathcal{C}}: \mathcal{C} \rightarrow \mathcal{C}^{\euler{K}}$.
 Using the definitions of $\Psi_{\mathcal{C}}, \Phi_{\mathcal{C}}$, it is easy to verify that the following equations hold:
 \begin{equation}\label{EnvelopeNaturality}
  \Phi_{\mathcal{D}} \circ F = F^{\oplus} \circ \Phi_{\mathcal{C}} \text{ and }\Psi_{\mathcal{D}} \circ F = F^{\euler{K}} \circ \Psi_{\mathcal{C}},
 \end{equation}
 for any $\Bbbk$-linear functor $F: \mathcal{C} \rightarrow \mathcal{D}$.
 Denote the respective inclusion $2$-functors by $\mathbf{J}^{\oplus} \in [\mathbf{Cat}_{\Bbbk},\mathbf{Cat}_{\Bbbk}^{\oplus}]$ and $\mathbf{J}^{\oplus} \in [\mathbf{Cat}_{\Bbbk},\mathbf{Cat}_{\Bbbk}^{\euler{K}}]$. 
 Equation \eqref{EnvelopeNaturality} implies that we have $2$-transformations
 \[
  \Phi: \mathbb{1}_{\mathbf{Cat}_{\Bbbk}} \rightarrow \mathbf{J}^{\oplus} \circ (-)^{\oplus} \text{ and } \Psi: \mathbb{1}_{\mathbf{Cat}_{\Bbbk}} \rightarrow \mathbf{J}^{\euler{K}} \circ (-)^{\euler{K}}.
 \]
 \item \cite[Proposition 75, Theorem 113]{Ri}: Given $\mathcal{C} \in \on{Ob}\mathbf{Cat}_{\Bbbk}$ and $\mathcal{D} \in \on{Ob}\mathbf{Cat}_{\Bbbk}^{\oplus}$, the functor
 \[
  \mathbf{Cat}_{\Bbbk}(\mathcal{C}^{\oplus},\mathcal{D}) \xrightarrow{\circ \Phi_{\mathcal{C}}} \mathbf{Cat}_{\Bbbk}(\mathcal{C},\mathcal{D})
 \]
 is an equivalence. Similarly, given $\mathcal{D} \in \mathbf{Cat}_{\Bbbk}^{\euler{K}}$, the functor
 \[
  \mathbf{Cat}_{\Bbbk}(\mathcal{C}^{\euler{K}},\mathcal{D}) \xrightarrow{\circ \Psi_{\mathcal{C}}} \mathbf{Cat}_{\Bbbk}(\mathcal{C},\mathcal{D})
 \]
 is an equivalence.
\end{itemize}

Let
$\mathbf{J}^{\euler{D}}: \mathbf{Cat}_{\Bbbk} \rightarrow \mathbf{Cat}_{\Bbbk}^{\euler{D}}$ denote the indicated inclusion $2$-functor. 
Using the above listed facts about $(-)^{\oplus}, (-)^{\euler{K}}$, we make the following conclusion:
\begin{proposition}\label{Biadjunction}
 The pairs $\left( (-)^{\oplus}, \mathbf{J}^{\oplus} \right)$ and $\left( (-)^{\euler{K}}, \mathbf{J}^{\euler{K}} \right)$ are bicategorically adjoint. Composing the adjunctions, we further conclude that the $2$-functors $((-)^{\euler{D}}, \mathbf{J}^{\euler{D}})$ are bicategorically adjoint.
\end{proposition}

\begin{proof}
 Given $\mathcal{C} \in \on{Ob}\mathbf{Cat}_{\Bbbk}$ and $\mathcal{D} \in \on{Ob}\mathbf{Cat}_{\Bbbk}^{\oplus}$, the $2$-naturality of the equivalence
 \[
  \mathbf{Cat}_{\Bbbk}(\mathcal{C}^{\oplus},\mathcal{D}) \xrightarrow{\circ \Phi_{\mathcal{C}}} \mathbf{Cat}_{\Bbbk}(\mathcal{C},\mathcal{D})
 \]
 follows from Equation \eqref{EnvelopeNaturality}. The $2$-naturality in $\mathcal{D}$ follows from the associativity of composition of functors, since the equivalences are given by precomposition. The proof for Karoubi envelope is completely analogous.
\end{proof}

\begin{remark}
 Observe that the bicategorical adjunction above is almost $2$-categorical, the only condition missing being that it is given by equivalences rather than isomorphisms of $\on{Hom}$-categories.
\end{remark}

Combining Proposition \ref{PseudoFCoco} with Proposition \ref{Biadjunction}, we obtain the following:
\begin{corollary}\label{SummaryCoCo}
 The $2$-categories $\mathbf{Cat}_{\Bbbk}^{\oplus}, \mathbf{Cat}_{\Bbbk}^{\euler{K}}, \mathbf{Cat}_{\Bbbk}^{\euler{D}}$ are bicategorically cocomplete. Given a $\Bbbk$-linear bicategory $\csym{B}$, the bicategories 
 \[
  [\csym{B},\mathbf{Cat}_{\Bbbk}^{\oplus}]_{\Bbbk}, \; [\csym{B},\mathbf{Cat}_{\Bbbk}^{\euler{K}}]_{\Bbbk}, [\csym{B},\mathbf{Cat}_{\Bbbk}^{\euler{D}}]_{\Bbbk}
 \]
 are bicategorically cocomplete. 
\end{corollary}

\begin{proof}
 The first part of the statement is an immediate consequence of the preceding results. The second part follows since the bicategorical adjunctions 
 \[
  ((-)^{\oplus}, \mathbf{J}^{\oplus}), \; ((-)^{\euler{K}}, \mathbf{J}^{\euler{K}}), \; ((-)^{\euler{D}}, \mathbf{J}^{\euler{D}})
 \]
 give bicategorical adjunctions
 \[
  \left([\csym{B},(-)^{\oplus}]_{\Bbbk},[\csym{B},\mathbf{J}^{\oplus}]_{\Bbbk}\right), \;  \left([\csym{B},(-)^{\euler{K}}]_{\Bbbk},[\csym{B},\mathbf{J}^{\euler{K}}]_{\Bbbk}\right), \;  \left([\csym{B},(-)^{\euler{D}}]_{\Bbbk},[\csym{B},\mathbf{J}^{\euler{D}}]_{\Bbbk}\right).
 \]
\end{proof}

Corollary \ref{SummaryCoCo} gives  two equivalent ways to compute colimits in $[\csym{B},\mathbf{Cat}_{\Bbbk}^{\euler{D}}]_{\Bbbk}$ (and the other two pseudofunctor bicategories considered): we may first compute the same colimit in $\mathbf{Cat}_{\Bbbk}$, apply the envelope $(-)^{\euler{D}}$ and lift it pointwise to a colimit in $[\csym{B},\mathbf{Cat}_{\Bbbk}^{\euler{D}}]_{\Bbbk}$, or we may compute the colimit in $\mathbf{Cat}_{\Bbbk}$, lift it pointwise to a colimit in $[\csym{B},\mathbf{Cat}_{\Bbbk}]_{\Bbbk}$ and then apply the envelope $(-)^{\euler{D}}$. 

Colimits in ordinary category theory give universal properties expressed in terms of unique factorization of morphisms. For bicategories, the universal properties are expressed in terms of equivalences of categories of $1$-morphisms. Following \cite{Ke2}, we call the resulting factorization of $1$-morphisms up to invertible $2$-morphisms the {\it one-dimensional aspect} of the universal property, and the resulting functorial bijections of $2$-morphisms the {\it two-dimensional aspect} of the universal property.
The next result further facilitates the computation of colimits. It can be shown by slightly modifying the proof of the analogous observation made in \cite[Section 3]{Ke1}, for non-enriched bicategories.
\begin{proposition}\label{OneDimAspect}
 Let $\csym{B}$ be bicategorically cocomplete. The one-dimensional aspect of the universal property of a weighted colimit in $\csym{B}$ implies its two-dimensional aspect.
\end{proposition}
More precisely, to obtain the $\Bbbk$-linear statement, one may argue directly on the level of cocones: given a colimiting cocone  $\mathbf{W} \star \mathbf{F}$ and a cocone $\mathbf{H}$ satisfying the one-dimensional universal property, one obtains $1$-morphisms $\mathbf{H} \rightarrow \mathbf{W} \star \mathbf{F}$ and $\mathbf{W} \star \mathbf{F} \rightarrow \mathbf{H}$. The respective one-dimensional universal properties then suffice to conclude that these $1$-morphisms are mutually quasi-inverse equivalences.

Ideally, we would like to further restrict our treatment of colimits to the $2$-subcategory $\mathfrak{A}_{\Bbbk}^{f}$ of $\mathbf{Cat}_{\Bbbk}^{\euler{D}}$, consisting of finitary categories. However, the condition of finite dimensional hom-spaces between objects is not preserved under taking weighted colimits - in fact, it fails already for conical colimits. 

This is very different from the setting of classical representation theory: the category $\mathbf{vec}_{\Bbbk}$ of finite dimensional vector spaces over $\Bbbk$ is abelian, and hence, in particular, cocomplete. As a consequence, if $\mathcal{C}$ is a $\Bbbk$-linear category, then the category $\mathcal{C}\!\on{-mod}$, defined as $\mathbf{Cat}_{\Bbbk}(\mathcal{C}, \mathbf{vec}_{\Bbbk})$, also is $\Bbbk$-linear and abelian. Our claim means that the categorification of this statement to finitary bicategories is false.

We now give an example of this phenomenon. Let $\mathbf{Cat}_{\Bbbk}^{\on{f.d.}}$ denote the $1,2$-full $2$-subcategory of $\mathbf{Cat}_{\Bbbk}$ whose objects are $\Bbbk$-linear categories with finite dimensional hom-spaces.

 \begin{example}
  
  Let $\Bbbk$ be a field. Consider the free $\Bbbk$-linear categories $\mathcal{A}_{2}, \mathcal{A}_{3}$ on the quivers
  \[
   A_{2} = \begin{tikzcd}[sep = small] {\euler{1}} \arrow[r, "a"] & {\euler{2}} \end{tikzcd}, \quad A_{3} = \begin{tikzcd}[sep = small] 1 \arrow[r, "b"] & 2 \arrow[r, "c"] & 3 \end{tikzcd}
  \]
 and the $\Bbbk$-linear functors $F_{b},F_{c}: \mathcal{A}_{2} \rightarrow \mathcal{A}_{3}$ determined by
 \[
  F_{b}(a) = b, \; F_{c}(a) = c.
 \]
 The coequalizer $\on{coeq}(F_{b},F_{c})$ in $\mathbf{Cat}_{\Bbbk}$ is given by the free $\Bbbk$-linear category on the quiver $\begin{tikzcd} X  \arrow[loop right]{r}{x} \end{tikzcd}$, together with the $\Bbbk$-linear functor $\mathcal{A}_{3} \xrightarrow{C_{b=c}} \on{coeq}(F_{b},F_{c})$ determined by 
 \[
  C_{b=c}(b) = C_{b=c}(c) = x.
 \]
 Assume that $\mathbf{Cat}_{\Bbbk}^{\on{f.d.}}$ admits a coequalizer 
 \[
 \mathcal{A}_{3} \xrightarrow{C_{b=c}^{\on{f.d.}}} \on{coeq}^{\on{f.d.}}(F_{b},F_{c})
 \]
 of the above functors. The universal property of $\on{coeq}(F_{b},F_{c})$ gives a $\Bbbk$-linear functor $K: \on{coeq}(F_{b},F_{c}) \rightarrow \on{coeq}^{\on{f.d.}}(F_{b},F_{c})$ such that $C_{b=c}^{\on{f.d.}} = K \circ C_{b=c}$. 

 Given $m \in \mathbb{Z}_{> 0}$, let $\mathcal{T}_{m}$ be the $\Bbbk$-linear category with a unique object $Y$ and a morphism $y \in \on{End}_{\mathcal{T}_{m}}(Y)$ such that $\on{End}_{\mathcal{T}_{m}}(Y) = \Bbbk[y]/\langle y^{m} \rangle$. Let $T_{m}: \mathcal{A}_{3} \rightarrow \mathcal{T}_{m}$ be the $\Bbbk$-linear functor determined by $T_{m}(b) = T_{m}(c) = y$. In particular, $T_{m}$ coequalizes $F_{b},F_{c}$. Applying the universal properties of the respective coequalizers, we obtain the following commutative diagram of $\Bbbk$-linear functors:
 \[
  \begin{tikzcd}
   \mathcal{A}_{3} \arrow[dr, swap, "T_{m}"] \arrow[r, "C_{b=c}"] & \on{coeq}(F_{b},F_{c}) \arrow[r, "K"] \arrow[d, "\widetilde{T}_{m}"] & \on{coeq}_{F_{b},F_{c}}^{\on{f.d.}} \arrow[dl, "\widetilde{T}_{m}^{\on{f.d.}}"] \\
   & \mathcal{T}_{m}
  \end{tikzcd}
 \]
 where $\widetilde{T}_{m}, \widetilde{T}_{m}^{\on{f.d.}}$ are obtained using the respective universal properties of their domains and hence are the unique $\Bbbk$-linear functors making the left inner triangle and the outer triangle commute, respectively.
 The functor $\widetilde{T}_{m}$ is determined by $\widetilde{T}_{m}(x) = y$. The right triangle part of the above diagram gives the following diagram of associative $\Bbbk$-algebras:
 \[
  \begin{tikzcd}
  \Bbbk[x] \arrow[d, two heads, "{x \mapsto y}"] \arrow[r, "K_{X,X}"] & \on{End}_{\on{coeq}_{F_{b},F_{c}}^{\on{f.d.}}}(K(X)) \arrow[dl, two heads, "(\widetilde{T}_{m}^{\on{f.d.}})_{K(X),K(X)}"] \\
   \Bbbk[y]/\langle y^{m} \rangle
  \end{tikzcd}
 \]
 Using image factorization, we replace $\on{End}_{\on{coeq}_{F_{b},F_{c}}^{\on{f.d.}}}(K(X))$ with its subalgebra generated by $K_{X,X}(x) =: z$. Since $\on{End}_{\on{coeq}_{F_{b},F_{c}}^{\on{f.d.}}}(K(X))$ is finite dimensional by assumption, the aforementioned subalgebra is a finite dimensional quotient of $\Bbbk[x]$, hence isomorphic to $\Bbbk[w]/\langle w^{k} \rangle$, for some $k$. Since $z$ is mapped to $y$ under $(\widetilde{T}_{m}^{\on{f.d.}})_{K(X),K(X)}$, we see that $z$ must be a radical element with nilpotency degree greater than or equal to $m$. This implies $k \geq m$. But, while we may vary $m$, the coequalizer, and thus also the integer $k$, remain unchanged, so the existence of a coequalizer in $\mathbf{Cat}_{\Bbbk}^{\on{f.d.}}$ would imply the existence of $k$ such that $k \geq m$ for all $m > 0$, which is a contradiction.
 \end{example}

In view of the Krull-Schmidt theorem for finitary categories, we see from Corollary \ref{SummaryCoCo} that colimits of finitary birepresentations, although themselves not necessarily finitary, may be computed starting from the indecomposable objects of the underlying categories. Formally, we have the following:
\begin{proposition}\label{IndecSuffice}
 Let $\mathcal{C}$ be a finitary category. Let $\mathcal{C}\!\on{-indec}$ be the full subcategory of indecomposable objects of $\mathcal{C}$. Then $\mathcal{C} \simeq (\mathcal{C}\!\on{-indec})^{\oplus}$.
 
 Let $X$ be a set of objects of $\mathcal{C}$, and let $\mathcal{X}$ be the full subcategory of $\mathcal{C}$ satisfying $\on{Ob}\mathcal{X} = X$. Then $\on{add}X \simeq \mathcal{X}^{\euler{D}}$.
\end{proposition}

\begin{proof}
 Since $\mathcal{C}$ is additive, by \cite[Proposition 99]{Ri}, the functor $\Phi_{\mathcal{C}}: \mathcal{C} \rightarrow \mathcal{C}^{\oplus}$ is an equivalence. 
 Let $\mathrm{I}$ be the inclusion functor $\mathcal{C}\!\on{-indec}\hookrightarrow \mathcal{C}$. Since $\mathrm{I}$ is full and faithful, by \cite[Lemma 101]{Ri}, so is $\mathrm{I}^{\oplus}: (\mathcal{C}\!\on{-indec})^{\oplus} \rightarrow \mathcal{C}^{\oplus}$. The Krull-Schmidt theorem implies that all the isomorphism classes of objects of $\mathcal{C}$ are given by those of finite direct sums of objects in $\mathcal{C}\!\on{-indec}$. Thus, $\mathrm{I}^{\oplus}$ is essentially surjective, and hence an equivalence.
 
 Similarly, since $\on{add}X$ is additive and idempotent split, we have 
 \[
  (\on{add}X)^{\euler{D}} \simeq \on{add}X,
 \]
 the full and faithful functor $\mathcal{X} \rightarrow \on{add}X$ gives rise to $\mathcal{X}^{\euler{D}} \rightarrow (\on{add}X)^{\euler{D}}$, which also is full and faithful, and further also essentially surjective, since any object of $\on{add}X$ is isomorphic to a direct summand of a direct sum of objects in $X$.
\end{proof}

\section{Projective bimodules for star algebras}\label{s4}

\subsection{The bicategories \texorpdfstring{$\widetilde{\csym{B}}$ and $\csym{B}$}{B-tilde and B}}
Given $a,b \in \mathbb{Z}$ with $a \leq b$, we denote by $[a,b]$ the set $\setj{a, a+1, \ldots, b}$.

Let $\mathtt{S}_{n}$ denote the star graph on $n+1$ vertices. Label the unique internal node of $\mathtt{S}_{n}$ by $0$, and the leaves by $1,2,\ldots, n$. Let $\Lambda_{n}$ denote the zigzag algebra on $\mathtt{S}_{n}$, i.e. the quotient of the path algebra of
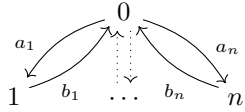
\[
 \begin{tikzcd}
  & 0 \arrow[dl, bend right=20, near end, swap, "a_{1}"] \arrow[d, shift left, dotted] \arrow[dr, bend left=20, near end, "a_{n}"] \\
  1 \arrow[ur, bend right=20, shift left = 0.9, swap, near start, "b_{1}"] & \cdots \arrow[u, dotted, shift left] & n \arrow[ul, bend left =20, shift right = 0.9, near start, "b_{n}"]
 \end{tikzcd}
\]
by the ideal given by the sum of the third power of the arrow ideal and the ideal given by relations $a_{j}b_{i} = 0$, for $i\neq j$, and $b_{1}a_{1} = \cdots = b_{n}a_{n} =: c$.
Let $e_{0},\ldots, e_{n}$ denote the \idem \ induced by the labelling on $\mathtt{S}_{n}$. For $k \in [1,n]$, denote by $c_{k}$ the element given by the $2$-cycle $a_{k}b_{k}$.

For a more extensive study of zigzag algebras, see for example \cite{ET}. 

 Consider the $\Bbbk$-linear monoidal subcategory $\widetilde{\csym{B}}_{n}$ of $(\Lambda_{n}\!\on{-mod-}\Lambda_{n}, \otimes_{\Lambda_{n}})$ given by the additive closure $\on{add}\setj{ {}_{\Lambda_{n}}\Lambda_{n}{}_{\Lambda_{n}}, \; \Lambda_{n} e_{k} \otimes_{\Bbbk} e_{0}\Lambda_{n} \; | \; k \in [0,n]}$. Viewed as a bicategory with a unique object, $\widetilde{\csym{B}}_{n}$ is a finitary bicategory. We denote its unique object by $\mathtt{i}$.
 In particular, $\widetilde{\csym{B}}_{n}$ is biequivalent to its essential image in $\mathbf{Cat}_{\Bbbk}(\Lambda_{n}\!\on{-mod},\Lambda_{n}\!\on{-mod})$, under the pseudofunctor sending a bimodule $M$ to the functor $M \otimes_{\Lambda_{n}} -$. Simple transitive $2$-representations of the $2$-category $\widetilde{\csym{B}}_{n}^{\on{str}}$ given by this essential image were studied in \cite{Zi}. 
 In this section, we use the partial results of \cite{Zi} to give a description of the underlying categories and action matrices for simple transitive birepresentations of $\widetilde{\csym{B}}_{n}$. As observed in \cite[2.3]{MMMTZ1}, studying the simple transitive birepresentations of $\widetilde{\csym{B}}_{n}$ is equivalent to studying the simple transitive $2$-representations of $\widetilde{\csym{B}}_{n}^{\on{str}}$. 
 
 To simplify the notation, we now fix $n \in \mathbb{Z}_{>0}$ and
 denote $\Lambda_{n}$ by $\Lambda$, denote $\widetilde{\csym{B}}^{\on{str}}_{n}$ by $\widetilde{\csym{B}}^{\on{str}}$, and finally write $\widetilde{\csym{B}}$ for $\widetilde{\csym{B}}_{n}$. We will introduce the subscripts again whenever there is a risk of ambiguity, in particular when different values of $n$ need to be considered simultaneously.
 
 Using the canonical isomorphism $\on{End}_{\Lambda\!\on{-mod}}(\Lambda) \simeq \Lambda^{\on{op}}$, we identify morphisms of indecomposable projective modules with elements of $\Lambda$. For example, given $i \in [1,n]$, this yields $\on{Hom}_{\Lambda\!\on{-proj}}(\Lambda e_{i},\Lambda e_{0}) = \Bbbk\setj{a_{i}}$. This applies also to the indecomposable bimodules $\Lambda e_{k} \otimes_{\Bbbk} e_{0}\Lambda$, where we identify morphisms with the images of the generators $ e_{k} \otimes e_{0}$. We obtain
\begin{equation}\label{CanonId}
  \on{Hom}_{\Lambda\!\on{-mod-}\!\Lambda}(\Lambda e_{j} \otimes_{\Bbbk} e_{0}\Lambda, \Lambda e_{k} \otimes_{\Bbbk} e_{0}\Lambda) = e_{j}\Lambda e_{k} \otimes_{\Bbbk} e_{0}\Lambda e_{0}.
\end{equation}
 Similarly, we identify morphisms from $\Lambda$ with the images of the generator $1$. An easy but tedious calculation yields
 \begin{equation}\label{HomsLambda}
 \begin{aligned}
  &\on{End}_{\Lambda\!\on{-mod-}\!\Lambda}(\Lambda) = \Bbbk\setj{1, c, c_{k} \; | \; k \in [1,n]}; \\
  &\on{Hom}_{\Lambda\!\on{-mod-}\!\Lambda}(\Lambda e_{j} \otimes_{\Bbbk} e_{0}\Lambda, \Lambda) = e_{j}\Lambda e_{0}; \\
  &\on{Hom}_{\Lambda\!\on{-mod-}\!\Lambda}(\Lambda, \Lambda e_{j} \otimes_{\Bbbk} e_{0}\Lambda) = \Bbbk\setj{b_{j} \otimes c + c_{j}\otimes b_{j}, c_{i} \otimes c}, \, j \neq 0; \\
  &\on{Hom}_{\Lambda\!\on{-mod-}\!\Lambda}(\Lambda, \Lambda e_{0} \otimes_{\Bbbk} e_{0}\Lambda) = \Bbbk\bigg\{e_{0} \otimes c + c \otimes e_{0} + \sum_{j=1}^{n}a_{j}\otimes b_{j}, c \otimes c\bigg\}.
 \end{aligned}
 \end{equation}
 Let $\mathcal{I}$ denote the ideal of $\on{add}\setj{ {}_{\Lambda}\Lambda {}_{\Lambda}, \; \Lambda e_{k} \otimes_{\Bbbk} e_{0}\Lambda \; | \; k \in [0,n]}$ determined by
 \begin{equation}\label{BiidealCalc}
 \begin{aligned}
  &\mathcal{I}(\Lambda, \Lambda) = \Bbbk\setj{c_{k} \; | \; k \in [1,n]}, \; \mathcal{I}(\Lambda e_{j} \otimes_{\Bbbk} e_{0}\Lambda, \Lambda) = 0, \\
  &\mathcal{I}(\Lambda, \Lambda e_{0} \otimes_{\Bbbk} e_{0}\Lambda) = 0, \; \mathcal{I}(\Lambda, \Lambda e_{j} \otimes_{\Bbbk} e_{0}\Lambda) = \Bbbk\setj{c_{i} \otimes c}, j \neq 0 \\
  &\mathcal{I}(\Lambda e_{j} \otimes_{\Bbbk} e_{0}\Lambda, \Lambda e_{k} \otimes_{\Bbbk} e_{0}\Lambda) =
  \Bbbk\setj{c_{j} \otimes e_{0}, c_{j} \otimes c} \text{ if } j = k \neq 0, \text{ otherwise }0.
 \end{aligned}
 \end{equation}

 \begin{lemma}\label{BiidealVer}
  The ideal $\mathcal{I}$ gives a biideal of $\widetilde{\csym{B}}$.
 \end{lemma}

 \begin{proof}
  It suffices to show that indecomposable non-identity $1$-morphisms send the $2$-morphisms given by $\mathcal{I}$ back to $\mathcal{I}$ under horizontal composition. Since, for any $v \in e_{0}\Lambda$ and any $j \in [1,n]$, we have $vc_{j} = 0$, it follows that 
  \[
   \on{id}_{\Lambda e_{j} \otimes_{\Bbbk} e_{0}\Lambda} \otimes \alpha = 0, \text{ for } \alpha \in \mathcal{I}.
  \]
  To see that $\mathcal{I}$ is also a right biideal, consider the case of the morphism
  \[
c_{i} \otimes e_{0} \in \mathcal{I}(\Lambda e_{0} \otimes_{\Bbbk} e_{0}\Lambda,\Lambda e_{0} \otimes_{\Bbbk} e_{0}\Lambda).
  \]
  Identifying isomorphic bimodules and identifying morphisms between decomposable bimodules with matrices of those given in \eqref{HomsLambda}, we may write
  \[
   (c_{i} \otimes e_{0}) \otimes \on{id}_{{\Lambda e_{j} \otimes_{\Bbbk} e_{0}\Lambda}} =
   \begin{cases}
    \left(\begin{smallmatrix} c_{i} \otimes e_{0} & 0 \\ 0 & c_{i} \otimes e_{0} \end{smallmatrix}\right) \text{ if } n = 0; \\
    c_{i} \otimes e_{0} \text{ otherwise.}
   \end{cases}
  \]
  from which it is clear that the resulting morphism remains in $\mathcal{I}$. The remaining cases are similar.
 \end{proof}
 Since $\mathcal{I}$ consists of radical morphisms, Lemma \ref{RadNilp} implies that it is nilpotent.
 Consider the unique cell birepresentation $\mathbf{C}$ of $\widetilde{\csym{B}}$, and the unique cell $2$-representation $\mathbf{C}^{\on{str}}$ of $\widetilde{\csym{B}}^{\on{str}}$. From the definition, we know that $\mathbf{C}(\mathtt{i})$ is the quotient of the category $\on{add}\setj{\Lambda e_{k} \otimes_{\Bbbk} e_{0}\Lambda \; | \; k \in [1,n]}$ by the unique maximal nilpotent ideal stable under the left $\widetilde{\csym{B}}$-action given by tensor products over $\Lambda$. Under the identification in \eqref{CanonId}, it corresponds to the maximal ideal $I$ of $\Lambda \otimes_{\Bbbk} e_{0}\Lambda e_{0}$ such that $e_{0}\Lambda \otimes_{\Lambda} I$ belongs to $\Lambda \otimes_{\Bbbk} \on{Rad}(e_{0}\Lambda e_{0})$. A simple calculation yields
 \[
I = \Bbbk\setj{c_{k} \; | \; k \in [1,n]} \otimes_{\Bbbk} e_{0}\Lambda e_{0} + \Lambda \otimes_{\Bbbk} \on{Rad}(e_{0}\Lambda e_{0}),
 \]
 which implies that for the ideal $J = \Bbbk\setj{c_{i} \; | \; i \in [1,n]}$, we have 
 \begin{equation}\label{CellIdentification}
\mathbf{C}(\mathtt{i}) \simeq \Lambda/J\!\on{-proj},\text{ under } \Lambda e_{k} \otimes_{\Bbbk} e_{0}\Lambda \mapsto (\Lambda/J)e_{k}.
 \end{equation}
 Let $A := \Lambda/J$. Remembering the index $n$, we obtain an algebra $A_{n}$, for every $n \in \mathbb{Z}_{>0}$. It follows that 
 \[
  \mathbf{C}(\Lambda e_{k} \otimes_{\Bbbk} e_{0}\Lambda) \simeq A e_{k} \otimes_{\Bbbk} e_{0}A \otimes_{A} -.
 \]
  Since $\mathbf{C}$ is equivalent to $\mathbf{C}^{\on{str}}$ as a birepresentation of $\widetilde{\csym{B}}$, we obtain 
  \begin{equation}\label{cellgives2fun}
\mathbf{C}^{\on{str}}(\Lambda e_{k} \otimes_{\Bbbk} e_{0}\Lambda \otimes_{\Lambda} -) \simeq A e_{k} \otimes_{\Bbbk} e_{0}A \otimes_{A} -.
  \end{equation}
 
 Let $\csym{B}$ denote the bicategory given by 
 \[
(\on{add}\setj{ {}_{A}A{}_{A}, \; A e_{k} \otimes_{\Bbbk} e_{0}A \; | \; k \in [0,n]}, -\otimes_{A}-).
 \]
 Consider its strictification $\csym{B}^{\on{str}}$, defined similarly to $\widetilde{\csym{B}}^{\on{str}}$. Remembering the index, we obtain the bicategories $\csym{B}_{n}$ and the $2$-categories $\csym{B}_{n}^{\on{str}}$. From \eqref{cellgives2fun}, we see that the assignments
 \[
  \mathtt{i} \mapsto \mathtt{i}, \; \mathrm{F} \mapsto \mathbf{C}^{\on{str}}\mathrm{F}, \; \alpha \mapsto \mathbf{C}^{\on{str}}\alpha
 \]
 give a $2$-functor
 \[
  \mathbf{Q}^{\on{str}}: \widetilde{\csym{B}}^{\on{str}} \rightarrow \csym{B}^{\on{str}}.
 \]
 Passing under biequivalences, this gives a pseudofunctor $\mathbf{Q}: \widetilde{\csym{B}} \rightarrow \csym{B}$. As described earlier, $\mathbf{Q}$ maps indecomposable $1$-morphisms to the corresponding indecomposable $1$-morphisms, and hence is essentially surjective. 
 
 Abusing notation, we identify elements of $\Lambda$ and their images of the projection from $\Lambda$ onto $A$, whenever such images are non-zero. A simple but tedious calculation yields
 \begin{equation}\label{HomsBigA}
 \begin{aligned}
  &\on{End}_{A\!\on{-mod-}\!A}(A) = \Bbbk\setj{1, c}; \\
  &\on{Hom}_{A\!\on{-mod-}\!A}(Ae_{j} \otimes_{\Bbbk} e_{0}A, A) = e_{j}Ae_{0}; \\
  &\on{Hom}_{A\!\on{-mod-}\!A}(A, Ae_{j} \otimes_{\Bbbk} e_{0}A) = \Bbbk\setj{b_{j} \otimes c}, \, j \neq 0; \\
  &\on{Hom}_{A\!\on{-mod-}\!A}(A, Ae_{0} \otimes_{\Bbbk} e_{0}A) = \Bbbk\bigg\{e_{0} \otimes c + c \otimes e_{0} + \sum_{j=1}^{n}a_{j}\otimes b_{j}, c \otimes c\bigg\}.
 \end{aligned}
 \end{equation}
 
 \begin{proposition}
  The pseudofunctor $\mathbf{Q}$ induces a biequivalence $\widetilde{\csym{B}}/\mathcal{I} \rightarrow \csym{B}$.
 \end{proposition}

 \begin{proof}
  $\mathcal{I}$ annihilates $\mathbf{C}^{\on{str}}$ by Lemma \ref{NilpotentAnn}, so, by construction, $\mathbf{Q}$ sends $\mathcal{I}$ to zero. There is thus an induced pseudofunctor $\widetilde{\mathbf{Q}}: \widetilde{\csym{B}}/\mathcal{I} \rightarrow \csym{B}$. It is essentially surjective, since $\mathbf{Q}$ is such. 
  Further, as a consequence of the calculations following Lemma~\ref{BiidealVer}, the cell birepresentation $\mathbf{C}$ is equivalent to the pullback birepresentation ${}_{\mathbf{Q}}A\!\on{-proj}$, where $A\!\on{-proj}$ has the structure of the defining birepresentation of $\csym{B}$. The latter is a faithful birepresentation (it is given by a locally faithful pseudofunctor) and so $\on{Ker}\mathbf{Q} = \mathcal{I}$. Indeed, if the kernel properly contained $\mathcal{I}$, its image under $\mathbf{Q}$ would give a non-zero biideal of $\csym{B}$ annihilating $A\!\on{-proj}$. Thus, $\widetilde{Q}$ is locally faithful.
  
  To see that it is also locally full, observe that, for bimodules $M,N \in \widetilde{\csym{B}}(\mathtt{i,i})$, using \eqref{HomsLambda},\eqref{BiidealCalc}, and \eqref{HomsBigA}, we obtain
  \[
   \on{dim}\on{Hom}_{\Lambda\!\on{-mod-}\Lambda}(M,N) = \on{dim}\on{Hom}_{A\!\on{-mod-}\!A}(\mathbf{Q}M, \mathbf{Q}N) + \on{dim}\mathcal{I}(M,N).
  \]
  The statement follows from the fact that an injective map of equidimensional finite dimensional vector spaces is an isomorphism.
 \end{proof}
\subsection{Prior results}
 Since $\mathcal{I}$ is nilpotent, Lemma \ref{NilpotentAnn} shows that simple transitive birepresentations of $\widetilde{\csym{B}}$ are the same as the simple transitive birepresentations of $\csym{B}$. Hence, all the analysis given in \cite{Zi} applies also if we make this replacement. On the level of the underlying $\Bbbk$-algebras, we replace $\Lambda$ by $A$. That the results of \cite{Zi} apply is even clearer if one observes the following:
 \begin{itemize}
  \item For $j \in [0,n]$, we have $e_{0}A e_{j} \simeq e_{0}\Lambda e_{j}$. Since $A$ is a quotient of $\Lambda$ by a nilpotent ideal, this implies that the multisemigroup given by composition of indecomposable $1$-morphisms is the same for $\widetilde{\csym{B}}$ and $\csym{B}$.
  \item  The module $A e_{0}$ is projective-injective, and so the bimodule $A e_{0} \otimes_{\Bbbk} e_{0}A$ gives a self-adjoint endofunctor of $A\!\on{-proj}$. 
  Following \cite[Remark 3.2]{Zi} one may apply \cite[Lemma 8]{MZ} to show that, for a simple transitive birepresentation $\mathbf{M}$ of $\csym{B}$, the functor $\mathbf{M}\mathrm{F}$ is a projective functor for any $1$-morphism $\mathrm{F}$.
 \end{itemize}
 Using these observations, one may prove \cite[Theorem 4.1, Theorem 5.1]{Zi} for $\csym{B}$ by verbatim repeating the proofs for $\widetilde{\csym{B}}$ given therein.
 
 From now on, the apex of any transitive birepresentation of $\csym{B}$ we consider is implicitly assumed to be the $J$-cell given by $\on{add}\setj{A e_{j} \otimes_{\Bbbk} e_{0}A \; | \; j \in [0,n]}$. Up to equivalence, there is a unique simple transitive birepresentation with apex given by $\on{add}\setj{{}_{A}A {}_{A}}$. This is a consequence of \cite[Theorem 18]{MM5}.
 
 Let $\mathfrak{P}$ be a set partition of a set $X$. We will associate a function to $\mathfrak{P}$, which, abusing notation, we denote by $\mathfrak{P}: X \rightarrow 2^{X}$. This function sends an element $x \in X$ to the subset $\mathfrak{P}(x)$ of $X$ it belongs to in the partition.
 \cite[Theorem 5.1]{Zi} implies that, for every simple transitive birepresentation $\mathbf{M}$ of $\csym{B}$, there is a set partition $\mathfrak{P}_{\mathbf{M}}$ of $[0,n]$ such that $[\mathbf{M}(Ae_{k} \otimes_{\Bbbk} e_{0}A)]_{\oplus} = [\mathbf{M}(Ae_{k'} \otimes_{\Bbbk} e_{0}A)]_{\oplus}$ if and only if $\mathfrak{P}(k) = \mathfrak{P}(k')$, and such that $\mathfrak{P}_{\mathbf{M}}(0) = \setj{0}$. 
 
 Completely in analogy to the proof of \cite[Theorem 4.1.]{Zi}, there is a strong transformation $\Sigma$ from the cell birepresentation $\mathbf{C}$ to $\mathbf{M}$, induced from the action of $\csym{B}$ on the simple top $L_{0}$ of the unique up to isomorphism object $Q_{0}$ of the essential image of $\overline{\mathbf{M}}(Ae_{0} \otimes_{\Bbbk} e_{0}A)$. Recall that $\mathbf{C}(\mathtt{i}) \simeq A\!\on{-proj}$, and, under that equivalence, we have $\mathbf{C}(Ae_{k} \otimes_{\Bbbk} e_{0}A) \simeq Ae_{k} \otimes_{\Bbbk} e_{0}A \otimes_{A} -$.
 
 For the remainder of this section, let $\mathbf{M}$ be a simple transitive birepresentation of $\csym{B}$, let $\Sigma$ be the above described strong transformation, and let $\mathfrak{P}_{\mathbf{M}}$ partition $[0,n]$ into $r+1$ subsets; recall that $\mathfrak{P}(0) = \setj{0}$. We now give a very short summary of the results of \cite{Zi} we will use.
 \begin{proposition}\label{JakobSummary}
  Given $k,k' \in [0,n]$, let $X,Y$ be indecomposable objects of $\mathbf{C}(\mathtt{i})$ belonging to the isomorphism classes of $A_{n}e_{k}, A_{n}e_{k'}$ in $A_{n}\!\on{-proj}$ under the identification $\mathbf{C}(\mathtt{i}) \simeq A_{n}\!\on{-proj}$ given in \eqref{CellIdentification}. 
  \begin{itemize}
   \item The objects $\Sigma(X), \Sigma(Y)$ are indecomposable, with $\Sigma(X) \simeq \Sigma(Y)$ if and only if $\mathfrak{P}_{\mathbf{M}}(k) = \mathfrak{P}_{\mathbf{M}}(k')$. 
   \item If $k=k'$ or $\mathfrak{P}(k) \neq \mathfrak{P}(k')$, then $\Sigma_{X,Y}$ gives a bijection of Hom-spaces. 
   \item As a consequence, if $U = \setj{u_{1},\ldots, u_{r}} \subseteq [0,n]$ is a transversal of $\mathfrak{P}_{\mathbf{M}}$, then the restriction of $\Sigma$ to $\on{add}\setj{A_{n}e_{k} \; | \; k \in U}$ is an equivalence of categories. In particular, $\mathbf{M}(\mathtt{i}) \simeq A_{r}\!\on{-proj}$.
  \end{itemize}
   
   We obtain an induced complete set of representatives of isomorphism classes of indecomposable objects of $\mathbf{M}(\mathtt{i})$, which we denote by $\setj{A_{r}^{\mathbf{M}}e_{\mathfrak{P}(u_{1})}, \ldots, A_{r}^{\mathbf{M}}e_{\mathfrak{P}(u_{r})}}$.
  Using $\mathbf{M}(\mathtt{i}) \simeq A_{r}\!\on{-proj}$ and identifying bimodules and endofunctors, we have
  \[
   \mathbf{M}(A_{n}e_{k} \otimes_{\Bbbk} e_{0}A_{n}) \simeq A_{r}^{\mathbf{M}}e_{\mathfrak{P}(u_{k})} \otimes_{\Bbbk} e_{\setj{0}}A_{r}^{\mathbf{M}}.
  \]
 \end{proposition}

 \begin{proof}
  For $n=1$, all the claims follow from the proof of \cite[Theorem 4.1]{Zi}. The exact same proof works for $n > 1$ after a modification of indices, since all the claims follow from the self-adjointness of $A_{n}e_{0} \otimes_{\Bbbk} e_{0}A_{n}$, which is independent of $n$, and the action matrices, which are completely determined by $\mathfrak{P}_{\mathbf{M}}$.
 \end{proof}
 
 Since action matrices are an invariant of finitary birepresentations, the following statement can be concluded already from their characterization in \cite[Theorem 5.1]{Zi}, but it is even clearer in view of Proposition \ref{JakobSummary}:
 \begin{corollary}\label{InvarianceOfPartitions}
  Let $\mathbf{M}, \mathbf{M}'$ be simple transitive birepresentations of $\csym{B}$. If $\mathfrak{P}_{\mathbf{M}} \neq \mathfrak{P}_{\mathbf{M}'}$, then $\mathbf{M} \not\simeq \mathbf{M}'$.
 \end{corollary}
 
 In particular, when classifying simple transitive $2$-representations of $\csym{B}$, we may do so for each set partition of $[0,n]$ separately.
 
 \subsection{Equifying modifications}
 Since $\csym{B}$ admits a unique object $\mathtt{i}$, we denote the principal birepresentation $\mathbf{P}_{\mathtt{i}} = \csym{B}(\mathtt{i},-)$ simply by $\mathbf{P}$. It admits a transitive subbirepresentation $\mathbf{N}$, given by the subcategory $\on{add}\setj{Ae_{k} \otimes_{\Bbbk} e_{0}A \; | \; k \in [0,n]}$ of $\csym{B}(\mathtt{i,i})$.

 For any birepresentation $\mathbf{M}$ of $\csym{B}$, Yoneda lemma gives 
 \[
  \csym{B}\!\on{-afmod}(\mathbf{P}, \mathbf{M}) \simeq \mathbf{M}(\mathtt{i}).
 \]
 Given an object $X \in \mathbf{M}(\mathtt{i})$, we denote by $\Theta_{X}: \mathbf{P} \rightarrow \mathbf{M}$ the strong transformation that sends a $1$-morphism $\mathrm{F}$ to the object $\mathbf{M}\mathrm{F}(X)$. 
 Denote the strong transformation $\mathbf{N} \rightarrow \mathbf{P}$, given by inclusion, by $\Gamma$. Let 
 $\Theta_{j} := \Theta_{A_{n}e_{j} \otimes_{\Bbbk} e_{0}A_{n}}$ be the indicated strong transformation from $\mathbf{P}$ to $\mathbf{C}$. 
 The identification $\mathbf{C}(\mathtt{i}) \simeq A\!\on{-proj}$ of \eqref{CellIdentification} sends the object $Ae_{j} \otimes_{\Bbbk} e_{0}A$ of the quotient category to $Ae_{j}$. We may thus identify $\Theta_{j}$ with the functor $- \otimes_{A_{n}} A_{n}e_{j}$.
 
 \begin{lemma}\label{IsoActions}
  Given $j,k \in [1,n]$, there is an invertible modification 
  \[
   \euler{s}_{j,k}: \Theta_{j} \circ \Gamma \xiso \Theta_{k} \circ \Gamma.
  \]
 \end{lemma}
 \begin{proof}
  Under the identifications above, a modification $\euler{m}$ from $\Theta_{j} \circ \Gamma$ to $\Theta_{k} \circ \Gamma$ is given by a natural transformation
  $
   \euler{m}: - \otimes_{A} Ae_{j} \rightarrow - \otimes_{A} Ae_{k}
  $
  of functors from $\mathbf{N}(\mathtt{i}) = \on{add}\setj{A e_{k} \otimes_{\Bbbk} e_{0}A \; | \; k \in [0,n]}$ to $A\!\on{-proj}$, such that, for any $M,N \in \mathbf{N}(\mathtt{i})$, the diagram
 \begin{equation}\label{ModifAxiom}
  \begin{tikzcd}[column sep = large]
  \left(M \otimes_{A} N\right) \otimes_{A} Ae_{j} \arrow[r, "\euler{m}_{M \otimes_{A} N}"] \arrow[d, "\mathfrak{a}_{M,N,Ae_{j}}"] & \left(M \otimes_{A} N\right) \otimes_{A}Ae_{k} \arrow[d, "\mathfrak{a}_{M,N,Ae_{k}}"] \\
  M \otimes_{A} \left(N \otimes_{A} Ae_{j} \right) \arrow[r, "M \otimes_{A} \euler{m}_{N}"] & M \otimes_{A} \left(N \otimes_{A} Ae_{k}\right)
  \end{tikzcd}
 \end{equation}
 commutes, where $\mathfrak{a}$ denotes the associator. Let $\euler{c}_{M}^{j}: M \otimes_{A} Ae_{j} \xiso Me_{j}$ denote the canonical isomorphism given by $m \otimes a \mapsto ma$. Since $(M \otimes_{A} N)e_{j} = M \otimes_{A} Ne_{j}$, we have $\euler{c}_{M \otimes_{A} N}^{j} = (M \otimes_{A} \euler{c}_{N}^{j}) \circ \mathfrak{a}_{M,N,Ae_{j}}$. 
 
 For $N \in \mathbf{N}(\mathtt{i})$, the right $A$-module $N_{A}$ lies in $\on{add}\setj{e_{0}A}$, and thus, for $j \in [1,n]$, the map $Ne_{j} \xrightarrow{(\cdot a_{j})_{N}} Nc$ is an isomorphism of left $A$-modules, since it is such an isomorphism for $N = e_{0}A$. Let $\varphi_{N}^{j,k}$ denote the isomorphism $(\cdot a_{k})_{N}^{-1} \circ (\cdot a_{j})_{N}$. Since $\varphi_{N}^{j,k}$ is defined in terms of right action, we have $\varphi_{M \otimes_{A} N}^{j,k} = M \otimes_{A} \varphi_{N}^{j,k}$.
 We define $(\euler{s}_{j,k})_{N}$ as $(\euler{c}_{N}^{k})^{-1} \circ \varphi_{N}^{j,k} \circ \euler{c}_{N}^{j}$. From the earlier statements, it follows that the diagram
 \[
 \begin{tikzcd}[column sep = large, scale cd = 0.988]
  (M \otimes N) \otimes Ae_{j} \arrow[r, "\euler{c}_{M \otimes N}^{j}"] \arrow[d, "\mathfrak{a}_{M,N,Ae_{j}}"] & M \otimes Ne_{j} \arrow[r, "\varphi_{M \otimes N}^{j,k}"] \arrow[d, equal] & M \otimes Ne_{k} \arrow[r, "\left(\euler{c}_{M \otimes N}^{k}\right)^{-1}"] \arrow[d, equal] & (M \otimes N) \otimes Ae_{k} \arrow[d, "\mathfrak{a}_{M, N, Ae_{k}}^{-1}"] \\
  M \otimes (N \otimes Ae_{j}) \arrow[r, "M \otimes \euler{c}_{N}^{j}"] & M \otimes Ne_{j} \arrow[r, "M \otimes \varphi_{N}^{j,k}"] & M \otimes Ne_{k} \arrow[r, "M \otimes \left(\euler{c}_{N}^{k}\right)^{-1}"] & M \otimes (N \otimes Ae_{k})
 \end{tikzcd}
 \]
 commutes, which proves that $\euler{s}$ satisfies the axiom \eqref{ModifAxiom}. Further, $\euler{s}_{N}$ is natural in $N$, since it is defined in terms of the right $A$-action on $N$, which commutes with left $A$-module morphisms.
 \end{proof}

\begin{lemma}\label{IsoM}
 Let $i,i' \in [1,n]$ be such that $\mathfrak{P}_{\mathbf{M}}(i) = \mathfrak{P}_{\mathbf{M}}(i')$. The strong transformations $\Sigma \circ \Theta_{i}$ and $\Sigma \circ \Theta_{i'}$ are isomorphic.
\end{lemma}

\begin{proof}
 From the description of $\Sigma$ in Proposition \ref{JakobSummary}, we have 
 \[
\Sigma \circ \Theta_{i}(A) \simeq \Sigma(Ae_{i} \otimes e_{0}A) \simeq A_{r}^{\mathbf{M}}e_{\mathfrak{P}(i)} = A_{r}^{\mathbf{M}}e_{\mathfrak{P}(i')} \simeq \Sigma \circ \Theta_{i'}(A)
 \]
  The result follows by Yoneda lemma applied on $\mathbf{P}$.
\end{proof}

Similarly to the argument preceding \eqref{CellIdentification}, under the identification \eqref{CanonId}, the unique maximal $\csym{B}$-stable ideal of $\mathbf{N}$, corresponds to $A \otimes_{\Bbbk} \on{Rad}(e_{0}Ae_{0})$. Using the identification of $\Theta_{i}$ with $- \otimes_{A} Ae_{i}$, we see that $\Theta_{i}$ sends this ideal to zero, which implies that $\Theta_{i} \circ \Gamma$ factors canonically through the projection $\mathbf{N} \rightarrow \mathbf{C}$. 
Let $\widetilde{\Theta_{i} \circ \Gamma}: \mathbf{C} \rightarrow \mathbf{C}$ denote the resulting transformation. Since $\mathbf{C}$ is simple transitive, the strong transformation $\Sigma \circ \widetilde{\Theta_{i} \circ \Gamma}$ is faithful. Since $\Theta_{i}$ sends the isomorphism class represented by $Ae_{j} \otimes_{\Bbbk} e_{0}A$ to that represented by $Ae_{j}$, we see that $\Sigma \circ \widetilde{\Theta_{i} \circ \Gamma}$ sends indecomposables to indecomposables, as prescribed by $\mathfrak{P}_{\mathbf{M}}$. 
Thus, the underlying functor of $\Sigma \circ \Theta_{i} \circ \Gamma$ is determined by a faithful, $\Bbbk$-linear functor from $A_{n}\!\on{-proj}$ to $A_{r}\!\on{-proj}$, which maps indecomposable objects to indecomposable objects, as prescribed by $\mathfrak{P}_{\mathbf{M}}$. Let $\euler{T}$ denote the set of isomorphism classes of such functors. Recall that $\on{Rad}\on{End}_{A_{n}\!\on{-mod}}(A_{n}e_{0}) = \Bbbk\setj{c}$, independently of $n$ (abusing notation by identifying $c \in A_{n}$ for varying $n$). From the above description it follows that the restriction of a functor $F$ in $\euler{T}$ to that subspace corresponds to an endomorphism of $\Bbbk\setj{c}$, hence a scalar, which we denote by $\chi_{F}$.
\begin{lemma}\label{ZeroComponent}
 The map $\euler{T} \rightarrow \Bbbk\!\setminus\!\setj{0}$, sending $F$ to $\chi_{F}$, is a bijection. An automorphism $\tau$ of $F$ in $\euler{T}$ is determined uniquely by $\tau_{A_{n}e_{0}}$.
\end{lemma}

\begin{proof}
  As a consequence of Proposition \ref{IndecSuffice}, an additive functor of finitary categories is determined, up to natural isomorphism, by its restriction to the full subcategory of indecomposable objects. Further, a natural transformation between such functors is uniquely determined by its components indexed by indecomposable objects. 
  Since we do not consider or assume any monoidal or strict monoidal structure on our categories and functors, we may simplify further by replacing the domain and codomain categories by equivalent, skeletal categories $\mathcal{N}, \mathcal{R}$. We write $\on{Ob}\mathcal{N} = \setj{A_{n}e_{0},\ldots, A_{n}e_{n}}$ and $\on{Ob}\mathcal{R} = \setj{A_{r}e_{0},\ldots, A_{r}e_{r}}$. 
  Choose $F, F' \in \euler{T}$. Identifying $c \in A_{n}$ with $c \in A_{r}$, the map $F_{A_{n}e_{0},A_{n}e_{0}}$ corresponds to an algebra endomorphism of $\Bbbk[c]/\langle c^{2}\rangle$. 
  Since $F$ is faithful, it sends $c$ to $\chi_{F}c$ with $\chi_{F} \neq 0$. Assume $\chi_{F} \neq \chi_{F'}$, and let $\sigma: F \rightarrow F'$ be a natural transformation. We have $F(A_{n}e_{0}) = F'(A_{n}e_{0}) = A_{r}e_{0}$. We write $\sigma_{0}(e_{0}) = \sigma_{0}e_{0} + \sigma_{c}c$. Naturality implies the commutativity of
 \[
  \begin{tikzcd}
   A_{r}e_{0} \arrow[r, "c \mapsto \chi c"] \arrow[d, swap, "e_{0} \mapsto \sigma_{0}e_{0} + \sigma_{c}c"] & A_{r}e_{0} \arrow[d, "e_{0} \mapsto \sigma_{0}e_{0} + \sigma_{c}c"] \\
   A_{r}e_{0} \arrow[r, "c \mapsto \chi' c"] & A_{r}e_{0}
  \end{tikzcd}
 \]
 and $\sigma_{0} \neq 0$ implies $\chi = \chi'$. But $\sigma_{0} = 0$ implies $\sigma_{Ae_{0}} \in \on{Rad}\on{End}A_{r}e_{0}$, so $\sigma$ is not an isomorphism.
 
 For the remaining values of $i$, we have $F(A_{n}e_{i}) = A_{r}e_{\mathfrak{P}(i)} = F'(A_{n}e_{i})$, and we may write $F(a_{i}) = \lambda_{i}a_{\mathfrak{P}(i)}$ and $F(b_{i}) = \mu_{i}b_{\mathfrak{P}(i)}$, for $\Bbbk$-scalars $\mu_{i}, \lambda_{i}$. Functoriality gives $\lambda_{i}\mu_{i} = \chi_{F}$. Similarly for $F'$ and $\mu_{i}',\lambda_{i}'$. Further, a natural isomorphism $\sigma: F \rightarrow F'$ would yield a diagram
 \[
  \begin{tikzcd}[column sep = huge]
   A_{r}e_{\mathfrak{P}(i)} \arrow[r, "e_{\mathfrak{P}(i)} \mapsto \lambda_{i}a_{\mathfrak{P}(i)} "] \arrow[d, swap, "e_{\mathfrak{P}(i)} \mapsto \sigma_{i}e_{\mathfrak{P}(i)}"] & A_{r}e_{0} \arrow[d, "e_{\mathfrak{P}(i)} \mapsto \sigma_{0}e_{0} + \sigma_{c}c"] \\
   A_{r}e_{\mathfrak{P}(i)} \arrow[r, "e_{\mathfrak{P}(i)} \mapsto \lambda'_{i}a_{\mathfrak{P}(i)} "] & A_{r}e_{0}
  \end{tikzcd}
 \]
and so, for a fixed $\sigma_{0}$, we may set $\sigma_{i} = \frac{\sigma_{0}\lambda_{i}}{\lambda'_{i}}$. Given $i, i' \in [1,n]$ such that $i \neq i'$, the Hom-space $\on{Hom}_{A_{n}\!\on{-proj}}(A_{n}e_{i}, A_{n}e_{i'})$ is zero, so the commutativity of the above diagram, together with the commutativity of 
 \[
  \begin{tikzcd}[column sep = huge]
   A_{r}e_{\mathfrak{P}(i)} \arrow[from =r, swap, "e_{0} \mapsto \frac{\chi}{\lambda_{i}}b_{\mathfrak{P}(i)}"] \arrow[d, swap, "e_{\mathfrak{P}(i)} \mapsto \sigma_{i}e_{\mathfrak{P}(i)}"] & A_{r}e_{0} \arrow[d, "e_{\mathfrak{P}(i)} \mapsto \sigma_{0}e_{0} + \sigma_{c}c"] \\
   A_{r}e_{\mathfrak{P}(i)} \arrow[from =r, swap, "e_{0} \mapsto \frac{\chi}{\lambda_{i}}b_{\mathfrak{P}(i)}"] & A_{r}e_{0}
  \end{tikzcd}
 \]
 suffice to conclude that $\sigma_{i} = \frac{\sigma_{0}\lambda_{i}}{\lambda_{i}'}$ does define a natural isomorphism. In particular, we see that $\sigma$ is completely determined by its component indexed by $A_{n}e_{0}$. This concludes the proof.
\end{proof}

\begin{lemma}\label{ProjFunctorFact}
 Let $B$ be a finite-dimensional $\Bbbk$-algebra and let $F$ be an indecomposable projective endofunctor of $B\!\on{-mod}$. Let $M \in B\!\on{-mod}, f,f' \in \on{End}_{B}(M)$ and let $X$ be an indecomposable summand of $FM$ such that the restrictions of $Ff, Ff'$ to $X$ are automorphisms. There is $\lambda \in \Bbbk\!\setminus\!\setj{0}$ such that $Ff'_{|X} = \lambda Ff_{|X}$.
\end{lemma}

\begin{proof}
 Clearly, if the statement holds for a functor $F$, it also holds for any functor $F'$ isomorphic to $F$. We may thus set $F = Be_{i} \otimes_{\Bbbk} e_{j}B \otimes_{B} -$ for some primitive idempotents $e_{i},e_{j} \in B$. Hence $FM = Be_{i} \otimes_{\Bbbk} e_{j}M \simeq Be_{i}^{\oplus \on{dim}e_{j}M}$, and $X \simeq Be_{i}$. For any $f \in \on{End}_{B}(M)$, the morphism $Ff$ maps $\Bbbk\setj{e_{i} \otimes v \; | \; v \in e_{j}M}$ to itself, and so $Ff$ cannot be a radical morphism. If $Ff_{|X}, Ff'_{|X}$ are linearly independent automorphisms of $X$, then, since the top of $X$ is simple, there is a linear combination 
 \[
  \mu Ff_{|X} + \mu' Ff'_{|X} = F(\mu f + \mu'f')_{|X} \in \on{Rad}\on{End}_{B}(X),
 \]
 which is a contradiction.
\end{proof}

\begin{lemma}\label{UniqueUpToScalar}
 Let $\euler{m}, \euler{m}': \Sigma \circ \Theta_{i} \circ \Gamma \rightarrow \Sigma \circ \Theta_{i} \circ \Gamma$ be a pair of invertible modifications. There is $\lambda \in \Bbbk\!\setminus\!\setj{0}$ such that $\euler{m}' = \lambda \euler{m}$.
\end{lemma}

\begin{proof}
 From Lemma \ref{ZeroComponent} and the discussion preceding it, we conclude that it suffices to show that $\euler{m}'_{A_{n}e_{0} \otimes_{\Bbbk} e_{0}A_{n}} = \lambda \euler{m}_{A_{n}e_{0} \otimes_{\Bbbk} e_{0}A_{n}}$. 
 
Since $\euler{m}$ is a modification, for every $1$-morphism $\mathrm{F} \in \csym{B}(\mathtt{i,i})$ and any object $X \in \mathbf{N}(\mathtt{i})$, we have
 \[
  \begin{tikzcd}
   \mathbf{M}(\mathrm{F}) \circ (\Sigma \circ \Theta_{i} \circ \Gamma)(X) \arrow[rr, "\mathbf{M}(\mathrm{F})\left((\euler{m})_{X}\right)"] \arrow[d, "(\Sigma \circ \Theta_{i} \circ \Gamma)_{X}"] & & \mathbf{M}(\mathrm{F}) \circ (\Sigma \circ \Theta_{i} \circ \Gamma)(X) \arrow[d, "(\Sigma \circ \Theta_{i} \circ \Gamma)_{X}"] \\ 
   (\Sigma \circ \Theta_{i} \circ \Gamma) \circ \mathbf{N}(\mathrm{F})(X) \arrow[rr, "(\euler{m})_{\mathbf{N}(\mathrm{F})(X)}"] & & (\Sigma \circ \Theta_{i} \circ \Gamma) \circ \mathbf{N}(\mathrm{F})(X)
  \end{tikzcd}
 \]
 and similarly for $\euler{m}'$.
 Let $\mathrm{F} = Ae_{0} \otimes_{\Bbbk} e_{0}A = X$. We have 
 \[
(Ae_{0} \otimes_{\Bbbk}e_{0}A)^{\otimes 2} := (Ae_{0} \otimes_{\Bbbk} e_{0}A) \otimes_{A} (Ae_{0} \otimes_{\Bbbk} e_{0}A) \simeq (Ae_{0} \otimes_{\Bbbk} e_{0}A)^{\oplus 2}.
 \]
 We fix a split monomorphism $\iota: Ae_{0} \otimes_{\Bbbk} e_{0}A \rightarrow (Ae_{0} \otimes_{\Bbbk} e_{0}A) \otimes_{A} (Ae_{0} \otimes_{\Bbbk} e_{0}A)$ and a split epimorphism 
 $\pi: (Ae_{0} \otimes_{\Bbbk} e_{0}A) \otimes_{A} (Ae_{0} \otimes_{\Bbbk} e_{0}A) \rightarrow Ae_{0} \otimes_{\Bbbk} e_{0}A$ such that $\pi \circ \iota = \on{id}_{Ae_{0} \otimes_{\Bbbk} e_{0}A}$. Due to the naturality in $X$, we obtain the commutative diagram
 \[
  \begin{tikzcd}[scale cd = 0.74, sep = 4ex]
   \mathbf{M}(Ae_{0} \otimes_{\Bbbk} e_{0}A) \circ (\Sigma \circ \Theta_{i} \circ \Gamma)(Ae_{0} \otimes_{\Bbbk} e_{0}A) \arrow[rrrr, "\mathbf{M}(Ae_{0} \otimes_{\Bbbk}e_{0}A)\left((\euler{m})_{Ae_{0} \otimes_{\Bbbk} e_{0}A}\right)"]  & & & & \mathbf{M}(Ae_{0} \otimes_{\Bbbk} e_{0}A) \circ (\Sigma \circ \Theta_{i} \circ \Gamma)(Ae_{0} \otimes_{\Bbbk} e_{0}A) \arrow[d, "(\Sigma \circ \Theta_{i} \circ \Gamma)_{Ae_{0} \otimes_{\Bbbk} e_{0}A}"] \\ 
   (\Sigma \circ \Theta_{i} \circ \Gamma)((Ae_{0} \otimes_{\Bbbk} e_{0}A)^{\otimes 2}) \arrow[rrrr, "\euler{m}_{(Ae_{0} \otimes_{\Bbbk} e_{0}A)^{\otimes 2}}"]  \arrow[u, "(\Sigma \circ \Theta_{i} \circ \Gamma)_{Ae_{0} \otimes_{\Bbbk} e_{0}A}^{-1}"] & & & &  (\Sigma \circ \Theta_{i} \circ \Gamma)((Ae_{0} \otimes_{\Bbbk} e_{0}A)^{\otimes 2}) \arrow[d, "(\Sigma \circ \Theta_{i} \circ \Gamma)(\pi)"] \\
    (\Sigma \circ \Theta_{i} \circ \Gamma)(Ae_{0} \otimes_{\Bbbk} e_{0}A) \arrow[rrrr, "\euler{m}_{Ae_{0} \otimes_{\Bbbk} e_{0}A}"] \arrow[u, "(\Sigma \circ \Theta_{i} \circ \Gamma)(\iota)"] & & & &  (\Sigma \circ \Theta_{i} \circ \Gamma)(Ae_{0} \otimes_{\Bbbk} e_{0}A) \\
  \end{tikzcd}
 \]
 Thus the restriction of $\mathbf{M}(Ae_{0} \otimes_{\Bbbk}e_{0}A)\left((\euler{m})_{Ae_{0} \otimes_{\Bbbk} e_{0}A}\right)$ to the indecomposable projective summand $\on{Im}\left((\Sigma \circ \Theta_{i} \circ \Gamma)_{Ae_{0} \otimes_{\Bbbk} e_{0}A}^{-1} \circ (\Sigma \circ \Theta_{i} \circ \Gamma)(\iota)\right)$ gives the automorphism $\euler{m}_{Ae_{0} \otimes_{\Bbbk} e_{0}A}$. 
 The same holds for $\euler{m}'$. The fact that  $\euler{m}'_{Ae_{0} \otimes_{\Bbbk} e_{0}A} = \lambda \euler{m}_{Ae_{0} \otimes_{\Bbbk} e_{0}A}$, for some non-zero $\lambda$, follows from Lemma \ref{ProjFunctorFact}.
\end{proof}

\begin{corollary}\label{Equification}
 Given $j,k \in [1,n]$ such that $\mathfrak{P}_{\mathbf{M}}(j) = \mathfrak{P}_{\mathbf{M}}(k)$, there is an invertible modification $\euler{t}_{j,k}: \Sigma \circ \Theta_{j} \rightarrow \Sigma \circ \Theta_{k}$ such that $\euler{t}_{j,k} \bullet \Gamma = \Sigma \bullet \euler{s}_{j,k}$, for $\euler{s}_{j,k}$ defined in Lemma \ref{IsoActions}.
\end{corollary}

\begin{proof}
 Using Lemma \ref{IsoM}, choose $\euler{t}'_{j,k}: \Sigma \circ \Theta_{j} \xiso \Sigma \circ \Theta_{k}$. From Lemma \ref{UniqueUpToScalar} it follows that there is a non-zero scalar $\lambda$ such that $t_{j,k} \bullet \Gamma = \lambda \Sigma \bullet \euler{s}_{j,k}$. The result follows by letting $\euler{t}_{j,k} := \frac{1}{\lambda} \euler{t}_{j,k}'$.
\end{proof}

\section{Simple transitive birepresentations of \texorpdfstring{$\csym{B}_{n}$}{Bn}}\label{s5}

We will use the theory of bicategorical weighted colimits to construct and classify simple transitive birepresentations of $\csym{B}$. The bicategories in which we will consider colimits are $\mathbf{Cat}_{\Bbbk}$ and the pseudofunctor bicategory $[\csym{B},\mathbf{Cat}_{\Bbbk}]_{\Bbbk}$. Observe that both these bicategories are, in fact, $2$-categories. This simplifies the notions below.

Let $\csym{C}$ be a $2$-category. 
Let $\csym{J}_{1}$ be the $2$-category given by $\begin{tikzcd}[sep = small] \bullet \arrow[r, shift right] \arrow[r, shift left] & \bullet \end{tikzcd}$. In particular, $\csym{J}_{1}$ has no non-identity $2$-morphisms. Let $\mathbf{F}$ be the $2$-functor $\csym{J}_{1} \rightarrow \csym{C}$ given by 
$\begin{tikzcd}[sep = small] \mathtt{i} \arrow[r, shift right, swap, "\mathrm{F}"] \arrow[r, shift left, "\mathrm{G}"] & \mathtt{j} \end{tikzcd}$, for objects $\mathtt{i,j}$ and $1$-morphisms $\mathrm{F,G}$ of $\csym{C}$.

Let $\mathbf{W}: \csym{J}_{1}^{\on{op}} \rightarrow \mathbf{Cat}$ be the $2$-functor given by the diagram $\begin{tikzcd} \euler{Iso} & \euler{1} \arrow[l, shift left, "P_{1}"] \arrow[l, shift right, swap, "P_{2}"] \end{tikzcd}$, where $\euler{1}$ denotes the terminal category with a unique object and only its identity morphism, and $\euler{Iso}$ is the walking isomorphism category, with two objects $P_{1}, P_{2}$ and, as its only non-identity morphisms, two mutually inverse  morphisms between the two objects.

The {\it bicategorical coisoinserter} of $\mathrm{F},\mathrm{G}$ above is the bicategorical colimit $\mathbf{W} \star \mathbf{F}$. One may verify that it is given by an object $\mathtt{w}$ of $\csym{C}$ together with a $1$-morphism $\mathrm{W}: \mathtt{j} \rightarrow \mathtt{w}$ and an invertible $2$-morphism $\zeta: \mathrm{W} \circ \mathrm{F} \xiso \mathrm{W} \circ \mathrm{G}$, such that, given a $1$-morphism $\mathrm{H}: \mathtt{j} \rightarrow \mathtt{k}$ and an invertible $2$-morphism $\gamma: \mathrm{HF} \xiso \mathrm{HG}$, there is a $1$-morphism $\widehat{\mathrm{H}}: \mathtt{w} \rightarrow \mathtt{k}$ and an invertible $2$-morphism $\widehat{\gamma}: \mathrm{H} \xiso \widehat{\mathrm{H}}\mathrm{W}$, such that the diagram
\begin{equation}\label{UniPropCoIso}
 \begin{tikzcd}
 \mathrm{HF} \arrow[d, "\gamma"] \arrow[r, "\widehat{\gamma}\bullet \mathrm{F}"] & (\widehat{\mathrm{H}}\mathrm{W})\mathrm{F} \arrow[r, equal] & \widehat{\mathrm{H}}(\mathrm{W}\mathrm{F}) \arrow[d, "\widehat{\mathrm{H}}\zeta"]\\
 \mathrm{HG} \arrow[r, "\widehat{\gamma}\bullet \mathrm{G}"] & (\widehat{\mathrm{H}}\mathrm{W})\mathrm{G} \arrow[r, equal] & \widehat{\mathrm{H}}(\mathrm{WG})
 \end{tikzcd}
\end{equation}
 commutes. Further, the pair $(\widehat{\mathrm{H}}, \widehat{\gamma})$ is unique up to an invertible $2$-morphism  compatible with $\widehat{\gamma}$. This is the one-dimensional aspect of the universal property of $(\mathtt{w},\mathrm{W},\zeta)$. In our applications we will not need the two-dimensional aspect, and hence we omit describing it.
 
 Let $\csym{J}_{2}$ be the $2$-category given by the diagram $
  \begin{tikzcd}
   \bullet \arrow[r, shift left=6pt, ""{name=U, near start, below}, ""{name = UR, near end,below}] \arrow[r, shift right=6pt, swap, ""{name=D, near start, above}, ""{name = DR, near end, above}] & \bullet
   \arrow[Rightarrow, from=U, to=D, swap, start anchor ={[yshift =3pt]}, end anchor = {[yshift=-3.8pt]}]
   \arrow[Rightarrow, from=UR, to=DR, start anchor ={[yshift =3pt]}, end anchor = {[yshift=-3.8pt]}]
  \end{tikzcd}
$
 Let $\mathbf{F}$ be the $2$-functor from $\csym{J}_{2}$ to $\csym{C}$ given by the diagram
$
  \begin{tikzcd}
   \mathtt{i} \arrow[r, shift left=6pt, "\mathrm{F}", ""{name=U, near start, below}, ""{name = UR, near end,below}] \arrow[r, shift right=6pt, "\mathrm{G}", swap, ""{name=D, near start, above}, ""{name = DR, near end, above}] & \mathtt{j}
   \arrow[Rightarrow, from=U, to=D, swap, "\alpha", start anchor ={[yshift =3pt]}, end anchor = {[yshift=-3.8pt]}]
   \arrow[Rightarrow, from=UR, to=DR, "\beta", start anchor ={[yshift =3pt]}, end anchor = {[yshift=-3.8pt]}]
  \end{tikzcd}
$
in $\csym{C}$. Let $\mathbf{W}: \csym{J}_{2}^{\on{op}} \rightarrow \mathbf{Cat}$ be the $2$-functor given by the diagram
$
  \begin{tikzcd}
   \euler{Arr} & \euler{1} \arrow[l, shift left=6pt, ""{name=D, near start, below}, "Q_{2}", ""{name = DR, near end,below}] \arrow[l, shift right=6pt, "Q_{1}", swap, ""{name=U, near start, above}, ""{name = UR, near end, above}] 
   \arrow[Rightarrow, from=U, to=D, swap, start anchor ={[yshift =-1pt]}, end anchor = {[yshift=1.8pt]}]
   \arrow[Rightarrow, from=UR, to=DR, start anchor ={[yshift =-1pt]}, end anchor = {[yshift=1.8pt]}]
  \end{tikzcd}
$,
where $\euler{Arr}$ is the walking arrow category, with two objects $Q_{1}, Q_{2}$ and a unique morphism $Q_{1} \rightarrow Q_{2}$. In particular, the two $2$-cells in the last diagram coincide.

The {\it bicategorical coequifier} of $\alpha$ and $\beta$ above is the colimit $\mathbf{W} \star \mathbf{F}$. One may verify that it is given by an object $\mathtt{r}$ together with a $1$-morphism $\mathrm{R}: \mathtt{j} \rightarrow \mathtt{r}$ such that there are equivalences 
\[
 \csym{C}(\mathtt{r},\mathtt{k}) \xrightarrow{\circ \mathrm{R}} \csym{C}(\mathtt{j},\mathtt{k})^{\on{eqf}},
\]
strongly natural in $\mathtt{k}$. Here, $\csym{C}(\mathtt{j},\mathtt{k})^{\on{eqf}}$ denotes the full subcategory of $\csym{C}(\mathtt{j},\mathtt{k})$ given by $1$-morphisms $\mathrm{H}$ such that $\mathrm{H}\alpha = \mathrm{H}\beta$.

Let $\mathfrak{P}$ be a set partition of $[0,n]$ satisfying $\mathfrak{P}(0) = \setj{0}$, given by 
\[
 [0,n] = \setj{0} \sqcup \setj{i_{1}^{1}, i_{2}^{1},\ldots, i_{k_{1}}^{1}} \sqcup \cdots \sqcup \setj{i_{1}^{r}, i_{2}^{r},\ldots, i_{k_{r}}^{r}}.
\]
Similarly to the notation $[1,k] = \setj{1,\ldots,k}$, we use the interval notation for subsets given by $\mathfrak{P}$, writing for example $[i_{1}^{2}, i_{3}^{2}] := \setj{i_{1}^{2}, i_{2}^{2}, i_{3}^{2}}$.

We associate particular kinds of weighted colimits to $\mathfrak{P}$. Let $\csym{J}_{n}$ be the $2$-category
$\begin{tikzcd}
    \bullet
    \arrow[r, draw=none, "\raisebox{+1.5ex}{\vdots}" description]
    \arrow[r, bend left]
    \arrow[r, bend right]
    &
    \bullet
\end{tikzcd}$ with $n$ parallel $1$-morphisms and no non-identity $2$-morphisms. Let $\mathbf{W}(\mathfrak{P})$ be the $2$-functor from $\csym{J}_{n}$ to $\mathbf{Cat}$ given by 
$\begin{tikzcd}
    \euler{Iso}(\mathfrak{P}) & \euler{1}
    \arrow[l, draw=none, "\raisebox{+1.5ex}{\vdots}" description]
    \arrow[l, bend left, "i_{k_{r}}^{r}"]
    \arrow[l, bend right, swap, "i_{1}^{1}"]
\end{tikzcd}$
where $\euler{Iso}(\mathfrak{P})$ is the category presented by
\[
 \begin{tikzcd}[row sep = 3ex]
  i_{1}^{1} \arrow[rr, shift left, "\psi^{1}_{1}"] & & i_{2}^{1} \arrow[ll, shift left, "(\psi^{1}_{1})^{-1}"] \arrow[rr, shift left, "\psi_{2}^{1}"] & & \cdots \arrow[ll, shift left, "(\psi_{2}^{1})^{-1}"] \arrow[rr, shift left, "\psi_{k_{1}-1}^{1}"] & & i_{k_{1}}^{1} \arrow[ll, shift left, "(\psi_{k_{1}-1}^{1})^{-1}"] \\
  i_{1}^{2} \arrow[rr, shift left, "\psi^{2}_{1}"] & & i_{2}^{2} \arrow[ll, shift left, "(\psi^{2}_{1})^{-1}"] \arrow[rr, shift left, "\psi_{2}^{2}"] & & \cdots \arrow[ll, shift left, "(\psi_{2}^{2})^{-1}"] \arrow[rr, shift left, "\psi_{k_{2}-1}^{2}"] & & i_{k_{2}}^{2} \arrow[ll, shift left, "(\psi_{k_{2}-1}^{2})^{-1}"] \\
  \vdots & \vdots & & \vdots & & \vdots & \vdots \\
  i_{1}^{r} \arrow[rr, shift left, "\psi^{r}_{1}"] & & i_{2}^{r} \arrow[ll, shift left, "(\psi^{r}_{1})^{-1}"] \arrow[rr, shift left, "\psi_{2}^{r}"] & & \cdots \arrow[ll, shift left, "(\psi_{2}^{r})^{-1}"] \arrow[rr, shift left, "\psi_{k_{r}-1}^{r}"] & & i_{k_{r}}^{r} \arrow[ll, shift left, "(\psi_{k_{r}-1}^{r})^{-1}"] \\
 \end{tikzcd}
\]

We now return to the setting of the bicategory $\csym{B}_{n}$ studied in the previous sections. $\mathbf{F}^{\mathtt{W}}(\mathfrak{P})$ be the $2$-functor from $\csym{J}_{n}$ to $[\csym{B}_{n},\mathbf{Cat}_{\Bbbk}]_{\Bbbk}$ given by the diagram
\[
\begin{tikzcd}
    \mathbf{P}
    \arrow[r, draw=none, "\raisebox{+1.5ex}{\vdots}" description]
    \arrow[r, bend left, "\Theta_{i_{1}^{1}}"]
    \arrow[r, bend right, swap, "\Theta_{i_{k_{r}}^{r}}"]
    &
    \mathbf{C}
\end{tikzcd}
\]
The colimit $\mathbf{W}(\mathfrak{P})\star \mathbf{F}^{\mathtt{W}}(\mathfrak{P})$ can be obtained by iterating coisoinserters, and is given by a $\Bbbk$-linear pseudofunctor $\mathbf{C}^{\mathtt{W}}(\mathfrak{P})$, a strong transformation $\Omega^{\mathtt{W}}(\mathfrak{P})$ and a set of invertible modifications $\setj{\widehat{\euler{x}}_{l}^{m}: \Omega^{\mathtt{W}}(\mathfrak{P}) \circ \Theta_{i_{l}^{m}} \xiso \Omega^{\mathtt{W}}(\mathfrak{P}) \circ \Theta_{i_{l+1}^{m}} \; | \; l \in [i_{1}^{m},i_{k_{m}-1}^{m}], m \in [1,r]}$. 

The one-dimensional property is analogous to that of coisoinserters given above: given a $\Bbbk$-linear pseudofunctor $\mathbf{M}$, a strong transformation $\Upsilon: \mathbf{C} \rightarrow \mathbf{M}$ and a family of modifications $\euler{n}_{l}^{m}$ indexed as above, there is a $1$-morphism $\widehat{\Upsilon}: \mathbf{C}^{\mathtt{W}}(\mathfrak{P})$ such that diagrams analogous to \ref{UniPropCoIso} commute for all choices of indices. In particular we have $\widehat{\Upsilon} \circ \Omega^{\mathtt{W}}(\mathfrak{P}) \simeq \Upsilon$.

Recall the birepresentation $\mathbf{N}$ of $\csym{B}_{n}$, the strong transformation $\Gamma: \mathbf{N} \rightarrow \mathbf{P}$ and the invertible modifications $\euler{s}_{j,k}: \Theta_{j} \circ \Gamma \xiso \Theta_{k} \circ \Gamma$, for $j,k \in [1,n]$. 
For any $m \in [1,r]$ and any $l \in [i_{1}^{m},i_{k_{m}-1}^{m}]$ we now have the parallel invertible modifications $\widehat{\euler{x}}_{l}^{m} \bullet \Gamma$ and $\Omega^{\mathtt{W}}(\mathfrak{P}) \bullet \euler{s}_{i_{l}^{m},i_{l+1}^{m}}$, from $\Omega^{\mathtt{W}}(\mathfrak{P}) \circ \Theta_{i_{l}^{m}} \circ \Gamma$ to $\Omega^{\mathtt{W}}(\mathfrak{P}) \circ \Theta_{i_{l+1}^{m}} \circ \Gamma$. 

We consider the multiple coequifier $\mathbf{C} \xrightarrow{\Omega^{\mathtt{R}}(\mathfrak{P})} \mathbf{C}^{\mathtt{WR}}(\mathfrak{P})$ coequifying all such pairs. In other words, we obtain equivalences 
\[
[\csym{B}_{n},\mathbf{Cat}_{\Bbbk}](\mathbf{C}^{\mathtt{WR}}(\mathfrak{P}),\mathbf{M}) \xrightarrow{\circ \Omega^{\mathtt{R}}(\mathfrak{P})} [\csym{B}_{n},\mathbf{Cat}_{\Bbbk}]_{\Bbbk}^{\on{eqf}}(\mathbf{C}^{\mathtt{W}}(\mathfrak{P}), \mathbf{M}),
\]
 strongly natural in $\mathbf{M} \in [\csym{B}_{n}, \mathbf{Cat}_{\Bbbk}]_{\Bbbk}$, where $[\csym{B}_{n},\mathbf{Cat}_{\Bbbk}]_{\Bbbk}^{\on{eqf}}(\mathbf{C}^{\mathtt{W}}(\mathfrak{P}), \mathbf{M})$ is the category of strong transformations from $\mathbf{C}^{\mathtt{W}}(\mathfrak{P})$ to $\mathbf{M}$ which, upon whiskering, coequify all the parallel pairs above. We write $\euler{x}_{l}^{m} = \Omega^{\mathtt{R}}(\mathfrak{P})\bullet \widehat{\euler{x}}_{l}^{m}$ and $\Omega^{\mathtt{WR}}(\mathfrak{P}) := \Omega^{\mathtt{R}}\circ \Omega^{\mathtt{W}}$.

\begin{lemma}\label{SoughtUni}
 Let $\mathbf{M} \in [\csym{B}_{n},\mathbf{Cat}_{\Bbbk}]_{\Bbbk}$. 
 Given a strong transformation $\mathbf{C} \xrightarrow{\Upsilon} \mathbf{M}$ together with invertible modifications $\euler{m}_{l}^{m}: \Upsilon \circ \Theta_{i_{l}^{m}} \rightarrow \Upsilon \circ \Theta_{i_{l+1}^{m}}$, for all $m \in [1,r]$ and for all $l \in [i_{1}^{m}, i_{k_{m}-1}^{m}1]$, such that $\euler{m}_{l}^{m} \bullet \Gamma = \Upsilon \bullet \euler{s}_{i_{l}^{m},i_{l+1}^{m}}$, there is a strong transformation $\widetilde{\Upsilon}: \mathbf{C}^{\mathtt{WR}} \rightarrow \mathbf{M}$ such that $ \widetilde{\Upsilon} \circ \Omega^{\mathtt{WR}} \simeq \Upsilon$. 
 Further, $\widetilde{\Upsilon}$ is unique up to an invertible modification.
\end{lemma}

\begin{proof}
 The statement follows from the one-dimensional aspect of the universal property of $(\mathbf{C}^{\mathtt{WR}}(\mathfrak{P}),\Omega^{\mathtt{WR}}(\mathfrak{P}), \euler{x}_{l}^{m})$ obtained by combining the universal properties of the iterated coisoinserter $\mathbf{C}^{\mathtt{W}}(\mathfrak{P})$ and the coequifier $\mathbf{C}^{\mathtt{R}}(\mathfrak{P})$.
 Note that the statement of the lemma does not exhaust the one-dimensional universal property.
\end{proof}

We now describe the underlying category for  $\mathbf{C}^{\mathtt{W}}(\mathfrak{P})$. 
Since it is constructed pointwise, it suffices to find the coisoinserter of the underlying diagram in $\mathbf{Cat}_{\Bbbk}$. We will only determine the underlying category up to equivalence, and thus, using Proposition \ref{Biadjunction}, Corollary \ref{SummaryCoCo}, and Proposition \ref{IndecSuffice} we may restrict the domain category $\mathbf{P}(\mathtt{i})$ to its full subcategory $\mathbf{P}(\mathtt{i})\!\on{-indec}$ of indecomposable objects. 
We may also restrict the codomain to the common essential image of the underlying functors of $\Theta_{i}$, for $i \in [1,n]$, which is exactly $\mathbf{C}(\mathtt{i})\!\on{-indec}$. This is clear under the identification of $\mathbf{C}(\mathtt{i})$ with $A_{n}\!\on{-proj}$ and of $\Theta_{i}$ with $-\otimes_{A_{n}} A_{n}e_{i}$. 
We have thus reduced the problem of finding the underlying coisoinserter to that of finding $\mathbf{W}(\mathfrak{P}) \star \mathbf{D}$, where $\mathbf{D}$ is the diagram
\begin{equation}
\label{IndecDiagram}
\begin{tikzcd}
    \mathbf{P}(\mathtt{i})\!\on{-indec}
    \arrow[r, draw=none, "\raisebox{+1.5ex}{\vdots}" description]
    \arrow[r, bend left, "- \otimes_{A_{n}} A_{n}e_{i_{1}^{1}}"]
    \arrow[r, bend right, swap, "- \otimes_{A_{n}} A_{n}e_{i_{k_{r}}^{r}}"]
    &
    A_{n}\!\on{-indec.proj}.
\end{tikzcd}
\end{equation}

\begin{remark}
 We recommend the reader to first consider the case $n=2$ and $\mathfrak{P}_{0}$ given by $[0,2] = \setj{0} \sqcup \setj{1,2}$. In fact, the only aspect of the remaining arguments that does not carry over verbatim from that case to the general case is taken care of by the explicit description of the weight $\mathbf{W}(\mathfrak{P})$ above. We will see that, in the case of $\mathfrak{P}_{0}$, one only needs to, in a sense, adjoin a single isomorphism to the category. In the general case, the form of the weight $\mathbf{W}(\mathfrak{P})$ tells us to add as few isomorphisms as possible to still obtain the sought identifications of isomorphism classes. This considerably facilitates the next step (using coequifiers), which is to ensure that the result is finitary.
\end{remark}

Let $\mathcal{P}^{\mathtt{S}}$ be the skeletal subcategory of $\mathbf{P}(\mathtt{i})\!\on{-indec}$ with 
\[
 \on{Ob}\mathcal{P}^{\mathtt{S}} = \setj{A,Ae_{k} \otimes_{\Bbbk} e_{0}A \; | \; k \in [0,n]}
\]
and let $I_{\mathcal{P}^{\mathtt{S}}}$ be its inclusion functor. 
Further, consider the full subcategories $\mathcal{C}^{\mathtt{A}}, \mathcal{C}^{\mathtt{S}}$ of $A\!\on{-indec.proj}$ given by 
\[
\begin{aligned}
&\on{Ob}\mathcal{C}^{\mathtt{A}} = \setj{A \otimes_{A} Ae_{j}, (Ae_{k} \otimes_{\Bbbk} e_{0}A) \otimes_{A} Ae_{j} \; | \; k \in [0,n], j \in [1,n]} \text{ and} \\
&\on{Ob}\mathcal{C}^{\mathtt{S}} = \setj{Ae_{k} \; | \; k \in [0,n]}.
\end{aligned}
\]
Using $e_{0}A \otimes_{A} Ae_{j} \simeq e_{0}Ae_{j} = \Bbbk\setj{b_{j}}$ for $j \in [1,n]$, let $F^{\mathtt{A,S}}: \mathcal{C}^{\mathtt{A}} \rightarrow \mathcal{C}^{\mathtt{S}}$ be the $\Bbbk$-linear functor induced by $A$-module isomorphisms
\[
\begin{aligned}
 Ae_{k} \otimes_{\Bbbk} e_{0}Ae_{j} &\xiso Ae_{k} \text{ and } &A \otimes_{A} Ae_{j} \xiso Ae_{j}. \\
 e_{k} \otimes b_{j} &\mapsto e_{k} \quad &1 \otimes e_{j} \mapsto e_{j} \quad
\end{aligned}
\]
In particular, $F^{\mathtt{A,S}}((Ae_{k} \otimes_{\Bbbk} e_{0}A) \otimes_{A} Ae_{j}) = Ae_{k}$, $F^{\mathtt{A,S}}(A \otimes_{A} Ae_{j}) = Ae_{j}$, and $F^{\mathtt{A,S}}$ is an equivalence. The full images of $(-\otimes_{A} Ae_{j}) \circ I_{\mathcal{P}^{\mathtt{S}}}$, for $j \in [1,n]$, are contained in $\mathcal{C}^{\mathtt{A}}$, and so we may consider the diagram
\begin{equation}\label{reducedDiagram}
 \begin{tikzcd}[column sep = large]
    \mathcal{P}^{\mathtt{S}}
    \arrow[rr, draw=none, "\raisebox{+1.5ex}{\vdots}" description]
    \arrow[rr, bend left=13, "F^{\mathtt{A,S}}\circ (- \otimes_{A_{n}}e_{i_{1}^{1}})\circ I_{\mathcal{P}^{\mathtt{S}}}"]
    \arrow[rr, bend right=13, swap, "F^{\mathtt{A,S}}\circ (- \otimes_{A_{n}}e_{i_{1}^{1}})\circ I_{\mathcal{P}^{\mathtt{S}}}"]
    & &
    \mathcal{C}^{\mathtt{S}}
\end{tikzcd}
\end{equation}
in $\mathbf{Cat}_{\Bbbk}$. Since the inclusions of the subcategories $\mathcal{C}^{\mathtt{S}}$, $\mathcal{P}^{\mathtt{S}}$ are both equivalences of categories, Diagram \ref{IndecDiagram} and Diagram \ref{reducedDiagram} are equivalent as $2$-functors to $\mathbf{Cat}_{\Bbbk}$. We may thus compute the colimit of this diagram in order to obtain an underlying category equivalent to the colimit of diagram \ref{IndecDiagram}, and, similarly, the universal cones will be compatible under such equivalence. This again reduces the problem of finding the colimit $\mathbf{C}^{\mathtt{W}}(\mathfrak{P})$. To simplify the notation, let $F_{i} := F^{\mathtt{A,S}} \circ (- \otimes_{A_{n}} A_{n}e_{i}) \circ I_{\mathcal{P}^{\mathtt{S}}}$, for $i \in [1,n]$. It is easy to describe $F_{i}$ explicitly. For example, we have
\[
 \begin{tikzcd}
  \left(Ae_{1} \otimes e_{0}A \xrightarrow{e_{1} \otimes e_{0} \mapsto a_{1} \otimes 2e_{0}} Ae_{0} \otimes_{\Bbbk} e_{0}A \arrow[r, "F_{i}"]\right) & \left( Ae_{1} \xrightarrow{e_{1} \mapsto 2a_{1}} Ae_{0}\right)
 \end{tikzcd}.
\]
Observe that all objects of $\mathcal{P}^{\mathtt{S}}$ and $\mathcal{C}^{\mathtt{S}}$ are cyclic modules. Similarly to Equations \eqref{HomsLambda}, we use the images of the cyclic generators $e_{j} \otimes e_{0}$ and $1$ to denote the morphisms in these categories.

\begin{notation}
 Due to the abundance of subscripts and superscripts in our notation coming from the enumerations induced by $\mathfrak{P}$, in Proposition \ref{DescribeCoiso}, Lemma \ref{CoequifCalculation} and their respective proofs we deviate from the standard notation, denoting the components of a natural transformation as if they were the argument of a function, rather than using a subscript. In other words, given functors $\mathrm{F,G}: \mathcal{C} \rightarrow \mathcal{D}$, a natural transformation $\tau: \mathrm{F} \rightarrow \mathrm{G}$ and some $X \in \on{Ob}\mathcal{C}$, we denote the component of $\tau$ at $X$ by $\tau(X)$ rather than $\tau_{X}$.
\end{notation}

\begin{proposition}\label{DescribeCoiso}
 The weighted colimit $\mathcal{C}^{\mathtt{W}}$ of the diagram \eqref{reducedDiagram} is obtained from $\mathcal{C}^{\mathtt{S}}$ by freely adjoining invertible morphisms 
 \[
\xi_{l}^{m}(j): Ae_{j} \rightarrow Ae_{j} \text{ and } \xi_{l}^{m}(A): Ae_{i_{l}^{m}} \rightarrow Ae_{i_{l+1}^{m}},
 \]
 for every $j \in [0,n]$, $m \in [1,r]$ and $l \in [i_{1}^{m}, i_{k_{m}-1}^{m}]$, and then imposing the relations giving
 \begin{equation}\label{CoisoRelations}
 \begin{aligned}
   &c \xi_{l}^{m}(0) = \xi_{l}^{m}(0)c; \;
   b_{i_{l+1}^{m}}\xi_{l}^{m}(0) = \xi_{l}^{m}(A)b_{i_{l}^{m}}; \;
   a_{i_{l+1}^{m}}\xi_{l}^{m}(A) = \xi_{l}^{m}(0)a_{i_{l}^{m}}; \\
   & a_{k}\xi_{l}^{m}(k) = \xi_{l}^{m}(0)a_{k}, \;
   b_{k}\xi_{l}^{m}(0) = \xi_{l}^{m}(k)b_{k}, \text{ for } k \in [1,n] \text{ and }l,m \text{ as above}.
 \end{aligned}
 \end{equation}
 The universal cone $\mathcal{C}^{\mathtt{S}} \xrightarrow{W} \mathcal{C}^{\mathtt{W}}$ is given by inclusion.
\end{proposition}

\begin{proof}
 Recall that in order to satisfy the universal property, we need natural isomorphisms $\xi_{l}^{m}: W \circ F_{i_{l}^{m}} \xiso W \circ F_{i_{l+1}^{m}}$ which are universal in the sense of \eqref{UniPropCoIso}.
 
 Let $\mathcal{D}$ be a $\Bbbk$-linear category and let $G: \mathcal{C}^{\mathtt{S}} \rightarrow \mathcal{D}$ be a $\Bbbk$-linear functor. 
 Given $l,m$ as above, a natural isomorphism $\tau_{l}^{m}: G \circ F_{i_{l}^{m}} \xiso G \circ F_{i_{l+1}^{m}}$ consists of $n+2$ isomorphisms, indexed by the objects of $\mathcal{P}^{\mathtt{S}}$. We write $\tau_{l}^{m}(j)$ for $\tau_{l}^{m}(Ae_{j} \otimes_{\Bbbk} e_{0}A)$.
 
 For $j \in [0,n]$, we have
 \[
  \tau_{l}^{m}(j): (G \circ F_{i_{l}^{m}})(Ae_{j} \otimes_{\Bbbk} e_{0}A) = G(Ae_{j}) \xiso G(Ae_{j}) = (G \circ F_{i_{l+1}^{m}})(Ae_{j} \otimes_{\Bbbk} e_{0}A).
 \]
 We denote the remaining component by $\tau_{l}^{m}(A)$ and have
 \[
  \tau_{l}^{m}(A): (G \circ F_{i_{l}^{m}})(A) = G(Ae_{i_{l}^{m}}) \xiso G(Ae_{i_{l+1}^{m}}) = (G \circ F_{i_{l+1}^{m}})(A).
 \]
 Equations \ref{HomsBigA} explicitly describe the morphisms in $\mathcal{P}^{\mathtt{S}}$ and so the conditions for naturality of $\tau_{l}^{m}$ can be found using these. We claim that $\tau_{l}^{m}$ is natural if and only if it satisfies the equations
 \begin{equation}\label{XiIsNat}
 \begin{aligned}
   & G(a_{j})\tau_{l}^{m}(j) = \tau_{l}^{m}(0)G(a_{j}), \;
   G(b_{j})\tau_{l}^{m}(0) = \tau_{l}^{m}(j)G(b_{j}) \text{ for } j \in [1,n]; \\
   &G(c) \tau_{l}^{m}(0) = \tau_{l}^{m}(0)G(c); \;
   G(b_{i_{l+1}^{m}})\tau_{l}^{m}(0) = \tau_{l}^{m}(A)G(b_{i_{l}^{m}}); \\
   &G(a_{i_{l+1}^{m}})\tau_{l}^{m}(A) = \tau_{l}^{m}(0)G(a_{i_{l}^{m}}). 
 \end{aligned}
 \end{equation}
 We show that the last equation is necessary for the naturality of $\tau_{l}^{m}$ - the remaining cases are very similar, and, in fact, easier.
 
 Consider the morphism $z: A \rightarrow Ae_{0} \otimes_{\Bbbk} e_{0}A$ given by 
 \[
  z(1) = e_{0} \otimes c + c \otimes e_{0} + \sum_{j=1}^{n} a_{j} \otimes b_{j}
 \]
 Tensoring with $Ae_{i_{l}^{m}}$ from the right gives
 \[
  \begin{aligned}
   &A \otimes_{A_{n}} Ae_{i_{l}^{m}} \rightarrow (Ae_{0} \otimes_{\Bbbk} e_{0}A) \otimes_{A_{n}} Ae_{i_{l}^{m}} \\
   &1 \otimes_{A_{n}} e_{i_{l}^{m}} \mapsto (e_{0} \otimes_{\Bbbk} c + c \otimes_{\Bbbk} e_{0} + \sum_{j=1}^{n} a_{j} \otimes_{\Bbbk} b_{j}) \otimes_{A_{n}} e_{i_{l}^{m}} = (a_{i_{l}^{m}} \otimes_{\Bbbk} b_{i_{l}^{m}}) \otimes_{A_{n}} e_{i_{l}^{m}}
  \end{aligned}
 \]
 and, to find $F^{\mathtt{A,S}}(z)$, we identify $(a_{i_{l}^{m}} \otimes_{\Bbbk} b_{i_{l}^{m}}) \otimes_{A} e_{i_{l}^{m}}$ with $a_{i_{l}^{m}} \otimes_{\Bbbk} b_{i_{l}^{m}}$. By definition, $F^{\mathtt{A,S}}$ sends this morphism to the morphism $Ae_{i_{l}^{m}} \xrightarrow{a_{i_{l}^{m}}} Ae_{0}$. Thus, $F_{i_{l}^{m}}(z) = a_{i_{l}^{m}}$, and similarly, $F_{i_{l+1}^{m}}(z) = a_{i_{l+1}^{m}}$. Naturality of $\tau_{l}^{m}$ requires
 \[
  G(a_{i_{l+1}^{m}}) \circ \tau_{l}^{m}(A) = (G \circ F_{i_{l+1}^{m}})(z) \circ \tau_{l}^{m}(A) = \tau_{l}^{m}(0) \circ (G \circ F_{i_{l}^{m}})(z) = \tau_{l}^{m}(0) \circ G(a_{i_{l}^{m}})
 \]
 which proves the necessity of the last equation in Equations \ref{XiIsNat}.
 Comparing Equations \ref{XiIsNat} with Equations \ref{CoisoRelations}, we see that 
 \[
 \setj{\xi_{l}^{m}(A)}\cup \setj{\xi_{l}^{m}(j) \; | \; j \in [0,n]}
 \]
 gives a natural isomorphism $\xi_{l}^{m}: W \circ F_{i_{l}^{m}} \rightarrow W \circ F_{i_{l+1}^{m}}$. 
 
 Let $G: \mathcal{C}^{\mathtt{S}} \rightarrow \mathcal{D}$ be a $\Bbbk$-linear functor and let $(\tau_{l}^{m})$ be a collection of natural isomorphisms $\tau_{l}^{m}: G \circ F_{i_{l}^{m}} \xiso G \circ F_{i_{l+1}^{m}}$, indexed by $l,m$ as above.
 
 We may extend $G$ to $\widehat{G}: \mathcal{C}^{\mathtt{W}} \rightarrow \mathcal{D}$ by setting $\widehat{G}(\xi_{l}^{m}(X)) := \tau_{l}^{m}(X)$, for $X \in \on{Ob}\mathcal{P}^{\mathtt{S}}$. From the definition of $(\mathcal{C}^{\mathtt{W}}, W)$, it follows that $\widehat{G}$ is the unique functor which satisfies $\widehat{G} \circ W = G$ and $\tau_{l}^{m} = \widehat{G}\bullet \xi_{l}^{m}$, for all $l,m$.
 
 This shows the one-dimensional aspect of the universal property of $(\mathcal{C}^{\mathtt{W}}, W, (\xi_{l}^{m})_{l,m})$.  From Proposition \ref{OneDimAspect} it now follows that this triple indeed gives the required colimit.
\end{proof}

\begin{remark}
 We have only proved the universal property with respect to functors and collections of natural isomorphisms giving $2$-transformations from $\mathbf{W}(\mathfrak{P})$, rather than all strong transformations, as we generally should when considering a bicategorical weighted colimit. Further, since $\widehat{G}$ is unique (rather than unique up to a compatible natural isomorphisms), we obtain isomorphisms rather than just equivalences. In other words, the above determines the weighted $2$-categorical colimit. However, the bicategorical colimit we were looking for is a PIE colimit, and bicategorical PIE colimits are equivalent to $2$-categorical PIE colimits, as observed for example in \cite[Remark 1.10.29]{Ha}, so the $2$-categorical colimit above is also the corresponding bicategorical colimit. This is not necessarily the case for weighted colimits which are not PIE colimits.
\end{remark}

Recall the above auxiliary categories $\mathcal{P}^{\mathtt{S}}, \mathcal{C}^{\mathtt{A}}, \mathcal{C}^{\mathtt{S}}$ and the auxiliary functors $F^{\mathtt{A,S}}, F_{i}, I_{\mathcal{P}^{\mathtt{S}}}$. We further consider the additional auxiliary category $\mathcal{N}^{\mathtt{S}}$, which is the full subcategory  of $\mathcal{P}^{\mathtt{S}}$ with 
$
 \on{Ob}\mathcal{N}^{\mathtt{S}} = \setj{Ae_{j} \otimes_{\Bbbk} e_{0}A \; | \; j \in [0,n]}
$,
and also its inclusion functor $I_{\mathcal{N}^{\mathtt{S}}}$.

We will calculate the multiple coequifier of all pairs
\[
 (W \circ F^{\mathtt{A,S}}) \bullet s_{i_{l}^{m},i_{l+1}^{m}} \bullet I_{\mathcal{N}^{\mathtt{S}}} \text{ and } \widehat{\xi}_{l}^{m} \bullet I_{\mathcal{N}}^{\mathtt{S}} \text{, for } l,m \text{ as above,}
 \]
 of modifications from $W \circ F^{\mathtt{A,S}} \circ (- \otimes_{A_{n}} A_{n} e_{i_{l}^{m}}) \circ I_{\mathcal{N}^{\mathtt{S}}}$ to $W \circ F^{\mathtt{A,S}} \circ (- \otimes_{A_{n}} A_{n} e_{i_{l+1}^{m}}) \circ I_{\mathcal{N}^{\mathtt{S}}}$, which we denote by $\mathfrak{s}_{l}^{m}$ and $\mathfrak{x}_{l}^{m}$, respectively.
 
 Similarly to the previous calculation, this coequifier $\mathcal{C}^{\mathtt{W}} \xrightarrow{R} \mathcal{C}^{\mathtt{WR}}$ in $\mathbf{Cat}_{\Bbbk}$ gives the underlying category and the cone functor of $\mathbf{C}^{\mathtt{W}} \xrightarrow{\Omega^{\mathtt{R}}} \mathbf{C}^{\mathtt{WR}}$, on the level of indecomposable objects.

\begin{lemma}\label{CoequifCalculation}
 The above described coequifier $\mathcal{C}^{\mathtt{WR}}$ is obtained from $\mathcal{C}^{\mathtt{S}}$ by adjoining invertible morphisms $\xi_{l}^{m}(A): Ae_{i_{l}^{m}} \rightarrow Ae_{i_{l+1}^{m}}$, satisfying the equations
 \begin{equation}\label{CoequiRelations}
   b_{i_{l+1}^{m}} = \xi_{l}^{m}(A)b_{i_{l}^{m}}; \;
   a_{i_{l+1}^{m}}\xi_{l}^{m}(A) = a_{i_{l}^{m}} \text{ for } l \in [i_{1}^{m},i_{k_{m}-1}^{m}] \text{ and } m \in [1,r].
 \end{equation}
 The universal cone $\mathcal{C}^{\mathtt{W}} \xrightarrow{R} \mathcal{C}^{\mathtt{WR}}$ is the functor is given by
 \begin{itemize}
  \item sending $\mathcal{C}^{\mathtt{S}}$, viewed as a subcategory of $\mathcal{C}^{\mathtt{W}}$ under the embedding described in Proposition \ref{DescribeCoiso}, to itself, viewed as a subcategory of $\mathcal{C}^{\mathtt{WR}}$, as described above;
  \item for all $l,m$ as above and any $j \in [0,n]$, sending $\xi_{l}^{m}(j)$ to $\on{id}_{Ae_{j}}$ and sending $\xi_{l}^{m}(A)$ to itself.
 \end{itemize}
 In particular, $(\mathcal{C}^{\mathtt{WR}})^{\euler{D}} \simeq A_{r}\!\on{-proj}$, where $(-)^{\euler{D}}$ is the envelope considered in Section \ref{s3}.
\end{lemma}

\begin{proof}
 Since $\mathfrak{x}_{l}^{m}$ is just the restriction of $\xi_{l}^{m}$ to the subset of objects of $\mathcal{P}^{\mathtt{S}}$ excluding the regular bimodule $A$, we know its components. To find the components of $\mathfrak{s}_{l}^{m}$, let $j,j' \in [1,n]$ and $k \in [0,n]$, and recall the definition of $\euler{s}_{i_{l}^{m},i_{l+1}^{m}}$ given in Lemma \ref{IsoActions} to find that, for any $j \in [0,n]$ and any $x \in Ae_{j}$, we have
 \[
  \begin{tikzcd}
  (Ae_{j} \otimes_{\Bbbk} e_{0}A) \otimes_{A} Ae_{i_{l}^{m}} \arrow[r, "\euler{s}_{i_{l}^{m},i_{l+1}^{m}}"] & (Ae_{j} \otimes_{\Bbbk} e_{0}A) \otimes_{A} Ae_{i_{l+1}^{m}} \arrow[r, "F^{\mathtt{S}}"] & Ae_{j} \arrow[r, "W"] & Ae_{j} \\
  (x \otimes b_{i_{l}^{m}}) \otimes e_{i_{l}^{m}} \arrow[r, mapsto] & (x \otimes b_{i_{l+1}^{m}}) \otimes e_{i_{l+1}^{m}} \arrow[r, mapsto] & x
  \end{tikzcd}
 \]
 and the image is exactly the same if we omit $\euler{s}_{i_{l}^{m},i_{l+1}^{m}}$ from the diagram. We conclude that $\mathfrak{s}_{l}^{m}({Ae_{j} \otimes e_{0}A}) = \on{id}_{Ae_{j}}$. Observe that, after substituting $\xi_{l}^{m}(k) = \on{id}_{Ae_{k}}$ in \ref{CoisoRelations}, the equations still hold, so the functor described in the lemma is well-defined, and the assignment $\zeta_{l}^{m}(Ae_{k} \otimes_{\Bbbk} e_{0}A) := \on{id}_{Ae_{k}}$, $(\zeta_{l}^{m})(A) = \xi_{l}^{m}(A)$ gives an isomorphism coinciding with $\mathfrak{s}_{l}^{m}$ on $Ae_{j} \otimes_{\Bbbk} e_{0}A$, for $j \in [0,n]$.
 
 We conclude that a $\Bbbk$-linear functor $G: \mathcal{C}^{\mathtt{W}} \rightarrow \mathcal{D}$ coequifies $\mathfrak{s}_{l}^{m}$ and $\mathfrak{x}_{l}^{m}$ if and only if it sends $\xi_{l}^{m}(k)$ to $\on{id}_{G(Ae_{k})}$. It is now clear that there is a unique functor $\widehat{G}: \mathcal{C}^{\mathtt{WR}} \rightarrow \mathcal{D}$ such that $\widehat{G} \circ R = G$.
 If we view $\mathcal{C}^{\mathtt{WR}}$ as a subcategory of $\mathcal{C}^{\mathtt{W}}$, under the clear embedding, we may describe $\widehat{G}$ as the corresponding restriction of $G$.
That $\mathcal{C}^{\mathtt{WR}}$ is the coequifier now follows from Proposition \ref{OneDimAspect}.
 
 To see that $(\mathcal{C}^{\mathtt{WR}})^{\euler{D}} \simeq A_{r}\!\on{-proj}$, note that $\mathcal{C}^{\mathtt{WR}}$ is obtained from the category of indecomposable $A_{n}$-projectives by adjoining isomorphisms that do not increase the dimensions of the $\on{Hom}$-spaces, due to the relations given in the lemma and the fact that we only adjoin isomorphisms along the linear orders $[i_{1}^{m}, i_{k_{m}}^{m}]$, so no new automorphisms can be obtained by composing the adjoined isomorphisms.
\end{proof}

\begin{proposition}\label{Identifications}
 The $\Bbbk$-linear pseudofunctor $\mathbf{C}^{\mathtt{WR}}(\mathfrak{P})$ is a finitary birepresentation of $\csym{B}$, with $\mathbf{C}^{\mathtt{WR}}(\mathfrak{P})(\mathtt{i}) \simeq A_{r}\!\on{-proj}$. Further, it satisfies the following properties:
 \begin{itemize}
 \item The canonical strong transformation $\mathbf{C} \xrightarrow{\Omega^{\mathtt{WR}}(\mathfrak{P})} \mathbf{C}^{\mathtt{WR}}(\mathfrak{P})$ is faithful. 
 \item Given indecomposable objects $X, Y \in \mathbf{C}(\mathtt{i})$, the map $\Omega^{\mathtt{WR}}_{X,Y}(\mathfrak{P})$ is a bijection unless $X \not\simeq Y$ and $\Omega^{\mathtt{WR}}(\mathfrak{P})(X) \simeq \Omega^{\mathtt{WR}}(\mathfrak{P})(Y)$. 
 \item If $U$ is a transversal of $\mathfrak{P}$, then the restriction of $\Omega^{\mathtt{WR}}(\mathfrak{P})$ to $\on{add}\setj{Ae_{k} \; | \; k \in U}$ is an equivalence of categories.
 \end{itemize}
\end{proposition}

\begin{proof}
 As observed earlier, applying $(-)^{\euler{D}}$ to the diagram 
 \begin{equation}\label{IndecDiagramII}
 \mathcal{C}^{\mathtt{S}} \xrightarrow{W} \mathcal{C}^{\mathtt{W}} \xrightarrow{R} \mathcal{C}^{\mathtt{WR}}
 \end{equation}
 gives a diagram in $\mathbf{Cat}_{\Bbbk}^{\euler{D}}$ equivalent to that underlying 
 \[
\mathbf{C} \xrightarrow{\Omega^{\mathtt{W}}} \mathbf{C}^{\mathtt{W}}(\mathfrak{P}) \xrightarrow{\Omega^{\mathtt{R}}} \mathbf{C}^{\mathtt{WR}}(\mathfrak{P}).
 \]
Since $(-)^{\euler{D}}$ sends faithful functors to faithful functors, and the remaining statements concern indecomposable objects and are preserved under equivalences of categories, the result follows from the explicit description of Diagram \eqref{IndecDiagramII}.
\end{proof}

\begin{proposition}\label{CoeqisoIsSimple}
 The birepresentation $\mathbf{C}^{\mathtt{WR}}(\mathfrak{P})$ is simple transitive.
\end{proposition}

\begin{proof}
 Let $\mathbf{I}$ be a $\csym{B}$-stable ideal of $\mathbf{C}^{\mathtt{WR}}(\mathfrak{P})$. Consider the $\csym{B}$-stable ideal $\Omega^{\mathtt{WR}}(\mathfrak{P})^{-1}\mathbf{I}$ of $\mathbf{C}$, as defined in Lemma \ref{Preimages}. 
 Since $\mathbf{C}$ is simple transitive, this latter ideal is either zero, or coincides with all of $\mathbf{C}$. In the former case, it is immediate that $\mathbf{I} = 0$. 
 
 Since 
 Let $\mathbf{I} = \mathbf{C}$. Since $\Omega^{\mathtt{WR}}$ restricted to $\on{add}\setj{Ae_{k} \; | \; k \in U}$, where $U$ is a transversal of $\mathfrak{P}$, is an equivalence of categories, it follows that the functor $\Omega^{\mathtt{WR}}$ is surjective on morphisms - its essential image is all of $\mathbf{C}^{\mathtt{WR}}(\mathfrak{P})(\mathtt{i})$. It follows that 
 $\Omega^{\mathtt{WR}}(\mathfrak{P})^{-1}\mathbf{I} = \mathbf{C}$ implies $\mathbf{I} = \mathbf{C}^{\mathtt{WR}}(\mathfrak{P})$. We conclude that $\mathbf{I} = 0$ or $\mathbf{I} = \mathbf{C}^{\mathtt{WR}}(\mathfrak{P})$, which proves that $\mathbf{C}^{\mathtt{WR}}(\mathfrak{P})$ is simple transitive.
\end{proof}

\begin{theorem}\label{Classification}
 Let $\mathbf{M}$ be a simple transitive birepresentation of $\csym{B}$. We then have $\mathbf{M} \simeq \mathbf{C}^{\mathtt{WR}}(\mathfrak{P}_{\mathbf{M}})$.
\end{theorem}

\begin{proof}
 Write $\mathfrak{P}_{\mathbf{M}}$ as 
\[
 [0,n] = \setj{0} \sqcup \setj{i_{1}^{1}, i_{2}^{1},\ldots, i_{k_{1}}^{1}} \sqcup \cdots \sqcup \setj{i_{1}^{r}, i_{2}^{r},\ldots, i_{k_{r}}^{r}}.
\]
 Consider the strong transformation $\Sigma: \mathbf{C} \rightarrow \mathbf{M}$ described in Proposition \ref{JakobSummary}. Using Corollary \ref{Equification}, we choose, for $m \in [1,r]$ and $l \in [i_{1}^{m}, i_{k_{m}-1}^{m}]$, an invertible modification $\euler{t}_{l}^{m}$ from
 $\Sigma \circ \Theta_{i_{l}^{m}}$ to $\Sigma \circ \Theta_{i_{l+1}^{m}}$ such that $\euler{t}_{l}^{m} \bullet \Gamma = \Sigma \bullet \euler{s}_{i_{l}^{m},i_{l+1}^{m}}$.
 
 This choice provides a strong transformation $\widetilde{\Sigma}: \mathbf{C}^{\mathtt{WR}}(\mathfrak{P}) \rightarrow \mathbf{M}$, using the universal property of $\mathbf{C}^{\mathtt{WR}}(\mathfrak{P})$. In particular, $\widetilde{\Sigma} \circ \Omega^{\mathtt{WR}}(\mathfrak{P}) \simeq \Sigma$. Let $U$ be a transversal for $\mathfrak{P}$. Denote by $\Sigma_{|U}$ the restriction of $\Sigma$ to $\on{add}\setj{Ae_{k} \; | \; k \in U}$. Using similar notation for $\widetilde{\Sigma} \circ \Omega^{\mathtt{WR}}(\mathfrak{P})$, we obtain
 \[
  \widetilde{\Sigma} \circ \Omega^{\mathtt{WR}}(\mathfrak{P})_{|U} \simeq \Sigma_{|U}.
 \]
 But Proposition \ref{JakobSummary} and Proposition \ref{Identifications} prove that $\Sigma_{|U}$ and $\Omega^{\mathtt{WR}}(\mathfrak{P})_{|U}$ are equivalences of categories. It follows that $\widetilde{\Sigma}$ is an equivalence of categories, and thus also an equivalence of birepresentations.
\end{proof}

\begin{theorem}\label{TheMainTheorem}
 The map
 \begin{equation}
  \begin{aligned}
 \Big\{ 
  \begin{aligned}
  \text{Set partitions of } [1,n]
  \end{aligned}
&\Big\}
 \longrightarrow 
  \Big\{ 
  \begin{aligned}
  \text{Simple transitive birepresentations of } \csym{B}_{n}
  \end{aligned}
\Big\}/\simeq
 \\
 &\mathfrak{P} \longrightarrow  \mathbf{C}^{\mathtt{WR}}(\mathfrak{P})  
 \end{aligned}
  \end{equation}
  is a bijection.
\end{theorem}

\begin{proof}
 From Proposition \ref{CoeqisoIsSimple} we have that, for any $\mathfrak{P}$, the birepresentation $\mathbf{C}^{\mathtt{WR}}(\mathfrak{P})$ is simple transitive. From Corollary \ref{InvarianceOfPartitions} it follows that for different set partitions $\mathfrak{P},\mathfrak{P}'$, we have $\mathbf{C}^{\mathtt{WR}}(\mathfrak{P}) \not\simeq \mathbf{C}^{\mathtt{WR}}(\mathfrak{P}')$. Finally, Theorem \ref{Classification} shows that any simple transitive birepresentation $\mathbf{M}$ is equivalent to $\mathbf{C}^{\mathtt{WR}}(\mathfrak{P})$, for $\mathfrak{P} = \mathfrak{P}_{\mathbf{M}}$.
\end{proof}

\begin{proposition}
 Let $\mathfrak{P}$ be a set partition, let $\mathfrak{P}'$ be a refinement of $\mathfrak{P}$ and let $\mathfrak{P}''$ be a refinement of $\mathfrak{P}'$. There is a strong transformation $\widetilde{\Omega}(\mathfrak{P},\mathfrak{P}'): \mathbf{C}^{\mathtt{WR}}(\mathfrak{P}) \rightarrow \mathbf{C}^{\mathtt{WR}}(\mathfrak{P}')$ such that
 \[
  \Omega^{\mathtt{WR}}(\mathfrak{P}') \simeq \widetilde{\Omega}(\mathfrak{P},\mathfrak{P}') \circ \Omega^{\mathtt{WR}}(\mathfrak{P}).
 \]
  Further, we have
 \[
  \widetilde{\Omega}(\mathfrak{P},\mathfrak{P}'') \simeq \widetilde{\Omega}(\mathfrak{P}',\mathfrak{P}'') \circ \widetilde{\Omega}(\mathfrak{P},\mathfrak{P}')
 \]
\end{proposition}

\begin{proof}
 Recall the universal property of $\mathbf{C}^{\mathtt{WR}}(\mathfrak{P})$ described in Lemma \ref{SoughtUni}. Applying it on the strong transformation $\Omega^{\mathtt{WR}}(\mathfrak{P}')$, we obtain $\widetilde{\Omega}(\mathfrak{P},\mathfrak{P}')$ satisfying
 \[
  \Omega^{\mathtt{WR}}(\mathfrak{P}') \simeq \widetilde{\Omega}(\mathfrak{P},\mathfrak{P}') \circ \Omega^{\mathtt{WR}}(\mathfrak{P}).
 \]
  Further, we conclude that a strong transformation satisfying the above property is unique up to invertible modification.
  
  Similarly, we have a strong transformation $\widetilde{\Omega}(\mathfrak{P},\mathfrak{P}'')$, unique up to invertible modification, such that
  \[
   \widetilde{\Omega}(\mathfrak{P},\mathfrak{P}'') \circ \Omega^{\mathtt{WR}}(\mathfrak{P}) \simeq \Omega^{\mathtt{WR}}(\mathfrak{P}'').
   \]
   We have
   \[
   \widetilde{\Omega}(\mathfrak{P}',\mathfrak{P}'') \circ \widetilde{\Omega}(\mathfrak{P},\mathfrak{P}') \circ \Omega^{\mathtt{WR}}(\mathfrak{P}) \simeq \widetilde{\Omega}(\mathfrak{P}',\mathfrak{P}'') \circ \Omega^{\mathtt{WR}}(\mathfrak{P}') \simeq \Omega^{\mathtt{WR}}(\mathfrak{P}''),
  \]
  which, by the above described uniqueness of $\widetilde{\Omega}(\mathfrak{P},\mathfrak{P}'')$, proves that
 \[
  \widetilde{\Omega}(\mathfrak{P},\mathfrak{P}'') \simeq \widetilde{\Omega}(\mathfrak{P}',\mathfrak{P}'') \circ \widetilde{\Omega}(\mathfrak{P},\mathfrak{P}').
 \]
\end{proof}

\begin{remark}
 The above construction of finitary birepresentations using coequifiers and coisoinserters essentially relies on the properties of the bijection
 \[
  e_{0}Ae_{i} \xiso e_{0}Ae_{j}, \qquad b_{i} \mapsto b_{j}
 \]
 which may be viewed as induced by the corresponding automorphism of $A$ sending $e_{k}$ to $e_{k}$, for $k \neq i,j$, $e_{i}$ to $e_{j}$ and $e_{j}$ to $e_{i}$. One is led to speculate that the construction can be mimicked for any basic finite dimensional algebra $B$ with a fixed \idem\ $\setj{f_{1},\ldots, f_{s}}$ and an automorphism $\psi$ permuting $\setj{f_{1},\ldots, f_{s}}$ with a fixed point. However, this is not the case. Consider the quotient $B$ of the path algebra of
 \[
  \begin{tikzcd}[sep = small]
   2 \arrow[r, bend right, swap, "b_{2}"] & 0 \arrow[l, bend right,swap, "a_{2}"] \arrow[r, bend left, "a_{1}"] & 1  \arrow[l, bend left, "b_{1}"]
  \end{tikzcd} \text{ by } \mathcal{R}^{4}, \text{ where } \mathcal{R} \text{ is the arrow ideal.}
 \]
 The quiver automorphism fixing $0$ and swapping $1$ and $2$ gives an automorphism $\psi$ of $B$. Denote the induced \idem\ by $\setj{f_{0},f_{1},f_{2}}$. For $\psi$ to give a natural isomorphism $-\otimes_{B} Bf_{1} \xiso - \otimes_{B} Bf_{2}$ on $\on{add}\setj{f_{0}B}$ extending $f_{0}Bf_{1} \xrightarrow{\psi_{|f_{0}Bf_{1}}} f_{0}Bf_{2}$, the diagram
 \[
  \begin{tikzcd}
   f_{0}Bf_{1} \arrow[r, "b_{1}a_{1} \cdot -"] \arrow[d, "\psi"] & f_{0}Bf_{1} \arrow[d, "\psi"] \\
   f_{0}Bf_{2} \arrow[r, "b_{1}a_{1} \cdot -"] & f_{0}Bf_{2}
  \end{tikzcd}
 \]
 would have to commute. However, chasing the element $b_{1}$, we obtain
 \[
  \begin{tikzcd}
   b_{1} \arrow[r, mapsto] & b_{1}a_{1}b_{1} \arrow[dr, mapsto] \\
   b_{2} \arrow[r, mapsto] & b_{1}a_{1}b_{2} \arrow[r, "\neq", no head, swap] & b_{2}a_{2}b_{2}
  \end{tikzcd}
 \]
 One can also construct examples where the automorphism of the algebra does indeed give a natural isomorphism as above, which is further monoidal - and thus gives a modification of strong transformations. One example is the quotient $B'$ of the path algebra of
 
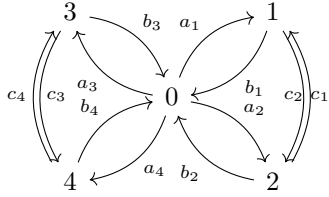
\[
  \begin{tikzcd}
   3 \arrow[dr, bend left, "b_{3}"] \arrow[dd, bend right, "c_{3}"] & & 1  \arrow[dd, bend left, "c_{1}"] \arrow[dl, bend left, "b_{1}"] \\
   & 0 \arrow[ur, bend left, "a_{1}"] \arrow[dr, bend left, "a_{2}"] \arrow[ul, bend left, "a_{3}"] \arrow[dl, bend left, "a_{4}"] \\
   4 \arrow[uu, bend left, shift left, "c_{4}"]  \arrow[ur, bend left, "b_{4}"] & & 2 \arrow[ul, bend left, "b_{2}"] \arrow[uu, bend right, shift right, "c_{2}"]
  \end{tikzcd}
 \]
 by the ideal generated by paths
 \[
 \begin{aligned}
  a_{3}b_{1}, a_{4}b_{1}, a_{3}b_{2}, a_{4}b_{2}, a_{1}b_{3},  a_{1}b_{4}, a_{2}b_{3}, a_{2}b_{4}, a_{k}b_{k}, b_{k}a_{k}, \text{ for } k \in [1,4]
 \end{aligned}
 \]
 together with the relations making all the $3$-cycles at a vertex coincide, and sending the fourth power of the arrow ideal to zero. Labelling the \idem\ coming from the quiver by $\setj{f_{0}, \ldots, f_{4}}$, we find that we have $\on{End}_{B'\!\on{-proj}}(B'f_{0}) \simeq \Bbbk[x]/\langle x^{2} \rangle$ and that $\on{Rad}\on{End}_{B'\!\on{-proj}}(B'f_{0})\otimes_{B'} B'f_{i} = 0$, for $i \neq 0$. 
 For the permutation $\sigma = (12)(34)$, the elements $x_{i} = a_{i} + c_{\sigma(i)}a_{\sigma(i)} \in f_{i}B'f_{0}$ are such that $f_{0}B'f_{i} \xrightarrow{\cdot x_{i}} f_{0}B'f_{0}$ is a $\Bbbk$-linear isomorphism. The quiver automorphism swapping $1,2$ and fixing the remaining vertices gives an automorphism $\psi_{(12)}$ of $B'$. 
 The resulting linear isomorphism $f_{0}B'f_{1} \xrightarrow{\psi} f_{0}B'f_{2}$ coincides with the composition $f_{0}B'f_{1} \xrightarrow{\cdot x_{1}} f_{0}B'f_{0} \xrightarrow{(\cdot x_{2})^{-1}} f_{0}B'f_{2}$. 
 These properties are completely analogous to those of star algebras, and using them one may verify that there is a finitary birepresentation of $(\on{add}\setj{{}_{B'}B'_{B'}, B'f_{k} \otimes_{\Bbbk} f_{0}B'},\otimes_{B'})$ given by an analogous (multiple) coequifier of (multiple) coisoinserter.
 
 However, for such examples, the analysis of the action matrices and Cartan matrices such as that for star algebras in \cite{Zi} has not been conducted, and so we cannot easily generalize our classification of simple transitive birepresentations to these cases.
 
\end{remark}

\vspace{5mm}

\noindent
Department of Mathematics, Uppsala University, Box. 480, SE-75106, Uppsala, SWEDEN,
email: {\tt mateusz.stroinski\symbol{64}math.uu.se}


\begin{thebibliography}{9999999999}
\bibitem[Br]{Br} M.~Brandenburg. Bicategorical colimits of tensor categories. Preprint, arXiv:2001.10123
\bibitem[Ca]{Ca} N.~Canevali. $2$-filtered bicolimits and finite weighted bilimits commute in {\bf Cat}. Bachelor's thesis, Universidad de Buenos Aires, 2016.
\bibitem[CM]{CM} A.~Chan, V.~Mazorchuk.   Diagrams and discrete extensions for finitary 2-representations.  Math. Proc. Cambridge Philos. Soc. {\bf 166} (2019), no. 2, 325--352.
\bibitem[CR]{CR} J.~Chuang, R.~Rouquier. Derived equivalences for symmetric groups and $\mathfrak{sl}_{2}$-categorification. Ann. of Math. (2) {\bf 167} (2008), no. 1, 245--298.
\bibitem[CW]{CW} J.~Comes, B.~Wilson. Deligne's category $\underline{\on{Rep}}(GL_{\delta})$ and representations of general linear supergroups. Represent. Theory, {\bf 16} (2012), 568--609.
\bibitem[De]{De} P.~Deligne. La cat{\' e}gorie des repr{\' e}sentations du groupe sym{\' e}trique $S_{t}$, lorsque $t$ n’est pas un entier naturel. Algebraic  groups  and  homogeneous  spaces, 209-273, Tata Inst. Fund. Res. Stud. Math., 19, Tata Inst. Fund. Res., Mumbai, 2007.
\bibitem[EGNO]{EGNO} P.~Etingof, S.~Gelaki, D.~Nikshych, V.~Ostrik. Tensor categories.
Mathematical Surveys and Monographs {\bf 205}. American Mathematical Society, Providence, RI, 2015.
\bibitem[ENO]{ENO} P.~Etingof, D.~Nikshych, V.~Ostrik. On fusion categories. Ann. of Math. (2) {\bf 162} (2005), no. 2, 581--642.
\bibitem[ET]{ET} M.~Ehrig, D.~Tubbenhauer. Algebraic properties of zigzag algebras. Comm. Algebra {\bf 48} (2020), no. 1, 11--36.
\bibitem[Fi]{Fi} T.~Fiore. Pseudo limits, biadjoints, and pseudo algebras: categorical foundations of conformal field theory. Mem. Amer. Math. Soc. {\bf 182} (2006), no. 860, x+171 pp.
\bibitem[GPS]{GPS} R.~Gordon, A.J.~Power, R.~Street. Coherence for tricategories. Mem. Amer. Math. Soc. {\bf 117} (1995), no. 558, vi+81 pp.
\bibitem[GS]{GS} R.~Garner, M.~Shulman. Enriched categories as a free cocompletion. Adv. Math. {\bf 289} (2016), 1-94.
\bibitem[Ha]{Ha} S.~Hazratpour. A logical study of some $2$-categorical aspects of topos theory. PhD thesis, University of Birmingham, 2019.
\bibitem[Jo1]{Jo1} H.~Jonsson. Cell structure of bimodules over radical square zero Nakayama algebras. Comm. Algebra {\bf 48} (2020), no. 2, 573--588.
\bibitem[Jo2]{Jo2} H.~Jonsson. On simple transitive 2-representations of bimodules over the dual numbers. Algebr. Represent. Theory {\bf 26}, 2057–2083 (2023)
\bibitem[JS]{JS} H.~Jonsson, M.~Stroi{\' n}ski. Simple transitive $2$-representations of bimodules over radical square zero Nakayama algebras via localization. J. Algebra {\bf 612} (2023) , 87–114.
\bibitem[Ke1]{Ke1} G.~M.~Kelly, Elementary Observations on $2$-categorical limits. Bull. Austral. Math. Soc. {\bf 39} (1989), no. 2, 301--317.
\bibitem[Ke2]{Ke2} G.~M.~Kelly, Basic Concepts of Enriched Category Theory.
 London Math. Soc. Lec. Note Ser. {\bf 64}. Cambridge U. Press, 1982.
\bibitem[KL1]{KL1} M.~Khovanov, A.~Lauda. A diagrammatic approach to categorification of quantum groups I. Represent. Theory {\bf 13} (2009), 309--347.
\bibitem[KL2]{KL2} M.~Khovanov, A.~Lauda. A categorification of quantum $\mathfrak{sl}_{n}$. Quantum Topol. {\bf 1} (2010), no. 1, 1--92.
\bibitem[KL3]{KL3} M.~Khovanov, A.~Lauda. A diagrammatic approach to categorification of quantum groups II. Trans. Amer. Math. Soc. {\bf 363} (2011), no. 5, 2685--2700.
\bibitem[Kh]{Kh} M.~Khovanov. A categorification of the Jones polynomial. Duke Math. J. {\bf 101} (2000), no. 3, 359--426.
\bibitem[Ma]{Ma} V.~Mazorchuk. Classification problems in $2$-representation theory. S{\~ a}o Paulo J. Math. Sci. 11 (2017), no.1, 1--22.
\bibitem[MM1]{MM1} V.~Mazorchuk, V.~Miemietz. Cell $2$-representations of finitary $2$-categories. Compositio Math. \textbf{147}  (2011), no. 5, 1519--1545.
\bibitem[MM2]{MM2} V.~Mazorchuk, V.~Miemietz. Additive versus abelian $2$-representations of
fiat $2$-categories. Mosc. Math. J. {\bf 14} (2014), no. 3, 595--615, 642.
\bibitem[MM3]{MM3} V.~Mazorchuk, V.~Miemietz. Endomorphisms of cell $2$-representations. Int. Math. Res. Not. IMRN  {\bf 2016}, no. 24, 7471--7498.
\bibitem[MM4]{MM4} V.~Mazorchuk, V.~Miemietz. Morita theory for finitary $2$-categories.
Quantum Topol. {\bf 7} (2016), no. 1, 1--28.
\bibitem[MM5]{MM5} V.~Mazorchuk, V.~Miemietz. Transitive $2$-representations of finitary $2$-categories.
Trans. Amer. Math. Soc. {\bf 368} (2016), no. 11, 7623--7644.
\bibitem[MM6]{MM6} V.~Mazorchuk, V.~Miemietz. Isotypic faithful $2$-representations of $\mathcal{J}$-simple
fiat $2$-categories. Math. Z. \textbf{282} (2016), no. 1-2, 411--434.
\bibitem[MMMT]{MMMT} M.~Mackaay, V.~Mazorchuk, V.~Miemietz, D.~Tubbenhauer. Simple transitive 2-representations
via (co)algebra $1$-morphisms. Indiana Univ. Math. J. {\bf 68} (2019), no. 1, 1--33.
\bibitem[MMMTZ1]{MMMTZ1} M. Mackaay, V. Mazorchuk, V. Miemietz, D. Tubbenhauer, X. Zhang. Finitary birepresentations of finitary bicategories. Forum Mathematicum {\bf 33} (2021), 1261-1320.
\bibitem[MMMTZ2]{MMMTZ2} M. Mackaay, V. Mazorchuk, V. Miemietz, D. Tubbenhauer, X. Zhang. Simple transitive 2-representations of Soergel bimodules for finite Coxeter types. . Proc. Lond. Math. Soc. (3) {\bf 126} (2023), no. 5, 1585–1655
\bibitem[MMZ]{MMZ} V.~Mazorchuk, V.~Miemietz, X.~Zhang. Pyramids and $2$-representations. Rev. Mat. Iberoam. \textbf{36} (2020), no. 2, 387--405.
\bibitem[MZ]{MZ} V.~Mazorchuk, X.~Zhang. Simple transitive 2-representations
for two non-fiat 2-ca\-te\-go\-ries of projective functors. Ukra{\" i}n. Math. J. 
{\bf 70} (2018), no. 12, 1625--1649.
\bibitem[Os]{Os} V.~Ostrik. Module categories, weak Hopf algebras and modular invariants. Transform. Groups {\bf 8} (2003), no. 2, 177-206.
\bibitem[Ri]{Ri} M.~Ritter. On universal properties of preadditive and additive categories. Bachelor's thesis, University of Stuttgart, 2016.
\bibitem[Ro]{Ro} R.~Rouquier. 2-Kac-Moody algebras. Preprint, arXiv: 0812.5023
\bibitem[St]{St} R.~Street. Fibrations in bicategories. Cahiers Topologie G{\' e}om. Diff{\' e}rentielle {\bf 21} (1980), no. 2, 111-160.
\bibitem[Zi]{Zi} J.~Zimmermann. Simple transitive 2-representations of left cell 2-subcategories of projective functors for star algebras. Comm. Algebra {\bf 47} (2019), no. 3, 1222--1237.
\end{thebibliography}
\end{document}